\documentclass{amsart}
\usepackage{amssymb}
\usepackage{amsthm}
\usepackage{amstext}
\usepackage{amsbsy}
\usepackage{amscd}
\usepackage{enumerate}
\usepackage{subfigure}
\usepackage[noadjust]{cite}
\usepackage[linktocpage, hidelinks]{hyperref}
\usepackage{graphicx}
\usepackage[usenames,dvipsnames]{color}

\newtheorem{thm}{Theorem}[section]

\newtheorem{lem}[thm]{Lemma}
\newtheorem{prop}[thm]{Proposition}

\newtheorem{claim}[thm]{Claim}

\theoremstyle{remark}
\newtheorem{rem}[thm]{Remark}
\theoremstyle{definition}

\newcommand{\norm}[1]{\left\Vert#1\right\Vert}
\newcommand{\abs}[1]{\left\vert#1\right\vert}
\newcommand{\set}[1]{\left\{#1\right\}}
\newcommand{\R}{\mathbb{R}}
\newcommand{\C}{\mathbb{C}}
\newcommand{\N}{\mathbb{N}}
\newcommand{\Z}{\mathbb{Z}}

\renewcommand{\dim}{\mathbf{dim}}
\newcommand{\hdim}{\dim_\mathrm{H}}
\newcommand{\bdim}{\dim_\mathrm{B}}
\newcommand{\loc}{\mathrm{loc}}
\newcommand{\lhdim}{\hdim^\loc}

\newcommand{\dist}{\mathbf{dist}}
\newcommand{\eqdef}{\overset{\mathrm{def}}=}

\newcommand{\p}{{\mathfrak p}}
\newcommand{\q}{{\mathfrak q}}

\newcommand{\T}{{\mathbb T}}
\newcommand{\Hd}{{\mathrm H}}
\newcommand{\Lyap}{{\mathrm{Lyap}}}

\usepackage{array}
\newcolumntype{L}{>{\arraybackslash}m{6cm}}

\begin{document}

\title[Mixed spectrum]{Mixed spectral regimes for square Fibonacci Hamiltonians}

\author[J. Fillman]{Jake Fillman}
\address{Mathematics, Rice University, 1600 Main St. MS-136, Houston, TX 77005}
\email{jdf3@rice.edu}

\author[Y. Takahashi]{Yuki Takahashi}
\address{Mathematics, University of California, Irvine, CA 92697}
\email{takahasy@uci.edu}

\author[W. Yessen]{William Yessen}
\address{Mathematics, Rice University, 1600 Main St. MS-136, Houston, TX 77005}
\email{yessen@rice.edu}

\thanks{J.\ F.\ was supported in part by NSF grants DMS--1067988 and DMS--1361625.}

\thanks{Y.\ T.\ was supported by the NSF grant DMS--1301515 (PI: A.\ Gorodetski).}

\thanks{\indent W.\ Y.\ was supported by the NSF grant DMS--1304287.}


\date{\today}

\begin{abstract}
For the square tridiagonal Fibonacci Hamiltonian, we prove existence of an open set of parameters which yield mixed interval-Cantor spectra (i.e.\ spectra containing an interval as well as a Cantor set), as well as mixed density of states measure (i.e.\ one whose absolutely continuous and singular continuous components are both nonzero). Using the methods developed in this paper, we also show existence of parameter regimes for the square continuum Fibonacci Schr\"odinger operator yielding mixed interval-Cantor spectra. These examples provide the first explicit examples of an interesting phenomenon that has not hitherto been observed in aperiodic Hamiltonians. Moreover, while we focus only on the Fibonacci model, our techniques are equally applicable to models based on any two-letter primitive invertible substitution.

\end{abstract}

\maketitle


\section{Introduction}

The purpose of this work is to investigate the interesting phenomenon of mixed interval-Cantor spectra occuring in some natural higher-dimensional versions of one-dimensional substitution models. More precisely, we study the square of the extensively studied Fibonacci Hamiltonian and its generalizations (compare \cite{Damanik2014e,Damanik2013X,Mei201x,DFG2014}). The reason for this choice is three-fold: (1) the one-dimensional versions have been extensively studied as prototypical examples of one-dimensional quasicrystals and the square models present a natural next step towards honest models of higher-dimensional quasicrystals; (2) the square models have also been considered in a physical setting \cite{Mandel2008b,Mandel2008,Mandel2006,Ilan2004}, and while numerical studies have appeared, analytical results are scarce (to the best of our knowledge, the only comprehensive work in this direction to date is \cite{Damanik2013X}); (3) while the more commonly accepted models of two-dimensional quasicrystals (such as those based on the Penrose tiling) are currently out of reach, the square models are amenable to a rigorous analysis using tools which are presently available. Our techniques are based on spectral theory, smooth dynamical systems, and geometric measure theory. Our techniques are equally applicable to models based on any two-letter primitive invertible substitution (compare \cite{Mei201x}).

\subsection{Models and main results}

Define the parameter space $\mathcal R$ by
\begin{align}\label{eq:params}
\mathcal{R} = \set{(\mathfrak{p},\mathfrak{q})\in\R^2: \mathfrak{p}\neq 0}.
\end{align}
Given $\theta \in \T \eqdef \R / \Z$, we define a Sturmian sequence $\set{\omega_n}$ via
\begin{align*}
\omega_n
\eqdef
\chi_{[1-\alpha, 1)}(n\alpha + \theta \mod 1),
\quad
n \in \Z,
\end{align*}
where $\alpha = \frac{\sqrt{5}-1}{2}$, the inverse of the golden mean. Given $\lambda = (\p,\q) \in \mathcal R$, we define the coefficients $\set{p_n}$ and $\set{q_n}$ by
\begin{equation} \label{eq:pqdef}
p_n
\eqdef
(\p-1)\omega_n + 1
=
\left\{
\begin{matrix}
1 & \text{ if } & \omega_n = 0\\
\mathfrak{p} & \text{ if } & \omega_n = 1\\
\end{matrix}\right.\hspace{2mm}
\text{ and }
\hspace{2mm}
q_n
\eqdef
\mathfrak{q}\omega_n,
\end{equation}
and then the operator $H_\lambda$ is defined on $\ell^2(\Z)$ via
\begin{align}\label{eq:haml}
(H_\lambda\phi)_n
=
p_{n+1}\phi_{n+1}+p_n\phi_{n-1}+q_n\phi_n.
\end{align}
Then the so-called \emph{square Hamiltonian}, which acts on $\ell^2(\Z^2)$, is given by
\begin{align}\label{eq:square-haml}
\begin{split}
\left( H^2_{(\lambda_1, \lambda_2)} \phi \right)_{n, m}=
(H_{\lambda_1}\phi(\cdot, m))_n + (H_{\lambda_2}\phi(n, \cdot))_m,
\end{split}
\end{align}
where $\phi(\cdot, m)$ and $\phi(n, \cdot)$ denote elements of $\ell^2(\Z)$ defined by $n\mapsto \phi(n, m)$ and $m\mapsto \phi(n, m)$, respectively. Equivalently, $\ell^2(\Z^2)$ is canonically isomorphic to the tensor product $\ell^2(\Z) \otimes \ell^2(\Z)$; under the canonical identification, $H_{(\lambda_1,\lambda_2)} \cong H_{\lambda_1} \otimes I_2 + I_1 \otimes H_{\lambda_2}$, where $I_j$ denotes the identity operator acting on the $j$th factor of the tensor product.

It is known that the spectrum of the operator $H_\lambda$ is independent of the choice of $\theta \in \T$, hence its suppression in the notation. The spectrum, the spectral measures, and the density of states measure of this operator have been studied in the context of electronic transport properties of quasicrystals, as well as their magnetic properties (see \cite{Yessen2011,Yessen2011a}, and the survey \cite{Mei201x}); this operator is also particularly attractive as a generalization of the extensively studied Fibonacci Schr\"odinger operator (see \cite{Damanik2013,Damanik2014e} and references therein). It is known that the spectrum of $H_\lambda$ is a Cantor set of zero Lebesgue measure whenever it is different from the free Schr\"odinger operator, i.e., whenever $\lambda \in \mathcal{R} \setminus \set{(1,0)}$ \cite[Theorem~2.1]{Yessen2011a}. Moreover, in this case, the spectral measures and the density of states measure are purely singular continuous. It is worth noting that in contrast to the Schr\"odinger Fibonacci Hamiltonian (i.e. the case in which $\mathfrak{p} = 1$), the spectrum of $H_\lambda$ exhibits rich multifractal structure in general (compare \cite[Theorem 1.1]{Damanik2014e} and \cite[Theorem 2.3]{Yessen2011a}).

\begin{rem}
Notice certain geometric restrictions in \cite[Theorem 2.3]{Yessen2011a}. Similar restrictions in the Schr\"odinger case were recently lifted in \cite{Damanik2014e}. In this paper, we lift those restrictions in full generality (Theorem \ref{thm:transversality} below), so that \cite[Theorem 2.3]{Yessen2011a} holds without any restrictions on the parameters.
\end{rem}

Like in the one-dimensional case, the interest in the square model is driven by the desire to understand fine quantum-dynamical properties of two-dimensional quasicrystals. Indeed, the square model was proposed as a model that is simpler than the commonly accepted model for a two-dimensional quasicrystal; unlike the general model, the square  model is susceptible to a rigorous study since its spectral data can be expressed in terms of the spectral data of the one-dimensional model. Moreover, like the general model, the square model is physically motivated (e.g. \cite{Mandel2006,Mandel2008,Mandel2008b,Ilan2004}). However, even this simplified case presents substantial challenges, and some of the resulting problems are of independent mathematical interest (e.g.\ sums of Cantor sets and convolutions of measures supported on them; see the introduction in \cite{Damanik2013X} for a deeper discussion and references).

Let us denote the spectrum of $H_\lambda$ by $\Sigma_\lambda$, and that of $H_{(\lambda_1, \lambda_2)}^2$ by $\Sigma_{(\lambda_1, \lambda_2)}^2$. By general spectral-theoretic arguments (see \cite[Appendix~A]{Damanik2013X}, for example),
\begin{align*}
\Sigma_{(\lambda_1, \lambda_2)}^2
=
\Sigma_{\lambda_1} + \Sigma_{\lambda_2}
\eqdef
\set{a + b: a \in\Sigma_{\lambda_1}, b\in\Sigma_{\lambda_2}},
\end{align*}
and the relevant measures supported on $\Sigma_{(\lambda_1, \lambda_2)}^2$ (i.e.\ the spectral and the density of states measures) are convolutions of the respective measures of the one-dimensional Hamiltonian.

Regarding the topology of $\Sigma_{(\lambda_1, \lambda_2)}^2$, some partial results are available. For example, if $\lambda_1 = (1, \mathfrak{q}_1)$ and $\lambda_2=(1, \mathfrak{q_2})$ with $\mathfrak{q}_i$, $i=1, 2$, sufficiently close to zero, $\Sigma_{(\lambda_1, \lambda_2)}^2$ is an interval, while for all $\mathfrak{q_i}$, $i = 1, 2$, sufficiently large, $\Sigma_{(\lambda_1, \lambda_2)}^2$ is a Cantor set of Hausdorff dimension strictly smaller than one, and hence also of zero Lebesgue measure \cite{Damanik2008,Damanik2010}. These properties are established via the dynamical properties of the associated trace map (we recall the basics in Section \ref{sec:background}; see \cite{Damanik2013} for details). Similar results hold for parameters of the form $(\mathfrak{p}, 0)$ (\cite[Appendix A]{Damanik2010} shows that all the spectral properties of the operators $H_{(1, \mathfrak{q})}$ that can be obtained from the trace map dynamics can also be obtained via the same methods, without modification, for the operator of the form $H_{(\mathfrak{p}, 0)}$). The topological picture in the intermediate regimes is currently unclear.

Let us adopt the following terminology for convenience: a set $A\subset\R$ is said to be \emph{interval-Cantor mixed} provided that $A$ has nonempty interior, and for some $a, b\in A$, $a < b$, $[a, b]\cap A$ is a (nonempty) Cantor set. In this paper, we prove

\begin{thm}\label{thm:thm2}
There exists a {\rm(}nonempty{\rm)} open set $\mathcal{U}\subset\mathcal{R}$ such that for every $\lambda_1, \lambda_2\in\mathcal{U}$, $\Sigma_{(\lambda_1, \lambda_2)}^2$ is interval-Cantor mixed.
\end{thm}

\begin{rem}\label{rem:thm1}
We suspect that for all $\lambda_1, \lambda_2\in\mathcal{R}\setminus\set{(1, 0)}$ sufficiently close to $(1, 0)$, $\Sigma_{(\lambda_1, \lambda_2)}^2$ is an interval; at present, however, our techniques only show that $\Sigma_{(\lambda_1, \lambda_2)}^2$ is a union of compact intervals. More details are given in Remark \ref{rem:thm2-details}.
\end{rem}

In \cite{Damanik2013X} the authors proved that for almost all pairs $(\mathfrak{q}_1, \mathfrak{q}_2)$ sufficiently close to $(0, 0)$, the convolution of the density of states measures of the Hamiltonians $H_{(1, \mathfrak{q}_1)}$ and $H_{(1, \mathfrak{q}_2)}$, supported on $\Sigma_{(1, \mathfrak{q}_1)} + \Sigma_{(1, \mathfrak{q}_2)}$, is absolutely continuous. In this paper, using the aforementioned multifractality of $\Sigma_\lambda$ and appealing to the techniques of \cite{Damanik2013X}, we prove the following measure-theoretic analog of Theorem \ref{thm:thm2}.

Denote the density of states measure for the Hamiltonian $H_{(\lambda_1, \lambda_2)}^2$ by
\begin{align}\label{eq:dos}
dk_{(\lambda_1, \lambda_2)}^2
=
dk_{\lambda_1} * dk_{\lambda_2},
\end{align}
where $dk_\lambda$ is the density of states measure for $H_\lambda$ (see, for example, \cite[Chapter~5]{Teschl1999} for definitions), and $*$ is the convolution of measures (see \cite[Appendix~A]{Damanik2013X} for the relation \eqref{eq:dos}). For $\lambda_1, \lambda_2\in\mathcal{R}$, let us denote by $(dk_{(\lambda_1, \lambda_2)}^2)_\mathrm{ac}$ and $(dk_{(\lambda_1, \lambda_2)}^2)_\mathrm{sc}$ the absolutely continuous and the singular continuous components (with respect to Lebesgue measure) of $dk_{(\lambda_1, \lambda_2)}^2$, respectively.

\begin{thm}\label{thm:thm1}
There exists a {\rm(}nonempty{\rm)} open set $\mathcal{U}\subset\mathcal{R}$ such that for every $\lambda_1\in\mathcal{U}$, there exists a full-measure subset $\mathcal V = \mathcal V(\lambda_1) \subseteq \mathcal U$ with the property that $(dk_{(\lambda_1, \lambda_2)}^2)_\mathrm{ac}\neq 0$ and $(dk_{(\lambda_1, \lambda_2)}^2)_\mathrm{sc}\neq 0$ whenever $\lambda_2 \in \mathcal V$.
\end{thm}

\begin{rem}\label{rem:thm2}
A nonempty open $\mathcal{U}\subset\mathcal{R}$ can be chosen so as to satisfy Theorems \ref{thm:thm2} and \ref{thm:thm1} simultaneously; see Remark \ref{rem:thm2-extension} below.

We suspect that one can choose $\mathcal U$ in such a way that the conclusion of Theorem~\ref{thm:thm1} holds for all pairs $\lambda_1,\lambda_2 \in \mathcal U$, not just a full-measure set, but our proof does not yield this stronger conclusion.
\end{rem}

As another application of the methods of this paper, we can also contruct separable two-dimensional continuum quasicrystal models which exhibit mixed topological structures in the spectrum. To construct an explicit example of a continuum quasicrystal model, take $\lambda > 0$ (note the change in parameter space) and $\theta \in \T$. Define $\omega_n$ as before, take
$$
V_\lambda(x)
=
\sum_{n \in \Z} \lambda \omega_n \chi_{[n,n+1)}(x),
$$
and define the self-adjoint operator $H_\lambda$ on $L^2(\R)$ by
\begin{align*}
H_\lambda \phi = -\phi'' + V_\lambda \phi.
\end{align*}

Models of this type were considered in \cite{DFG2014,LenzSeifStoll2014}. By the results therein, we know that $\sigma(H_\lambda)$ does not depend on $\theta$ and is a (noncompact) Cantor set of zero Lebesgue measure for all $\lambda > 0$. We denote this common spectrum by $\Sigma_\lambda$. As before, we define the square operator on $L^2(\R^2) \cong L^2(\R) \otimes L^2(\R)$ by $H_{(\lambda_1,\lambda_2)}^2 = H_{\lambda_1} \otimes I_2 + I_1 \otimes H_{\lambda_2}$; the spectrum of the square Hamiltonian is given by $\Sigma_{(\lambda_1, \lambda_2)}^2 = \Sigma_{\lambda_1} + \Sigma_{\lambda_2}$, as above.

\begin{thm}\label{thm:thm3}
For all $\lambda_1, \lambda_2 > 0$, the interior of $\Sigma_{(\lambda_1, \lambda_2)}^2$ is nonempty. Moreover, for all $\lambda_1, \lambda_2$ sufficiently large, there exists an interval $I_0$ such that $I_0 \cap \Sigma_{(\lambda_1, \lambda_2)}^2$ is a nonempty Cantor set.
\end{thm}

\begin{rem}\label{rem:thm3}
A few remarks are in order here.
\begin{enumerate}\itemsep0.5em

\item We suspect that more is true, namely, for all $\lambda_1, \lambda_2 > 0$, there exists $E_1 = E_1(\lambda_1), E_2 = E_2(\lambda_2) > 0$ such that $\Sigma_{(\lambda_1, \lambda_2)}^2$ contains the ray $[E_1 + E_2, \infty)$, but our methods do not prove this. If this is the case, the small-coupling behavior of $E_i$ is of interest -- for example, does one have $E_i(\lambda_i) \to 0$ as $\lambda_i \to 0$? Some details on this are contained in Remark~\ref{rem:thm3-details}.

\item It follows directly from the results of \cite{DFG2014} that for all $\lambda \geq 0$, the Hausdorff dimension of $\Sigma_\lambda$ is equal to one. It is natural to ask where in the spectrum the Hausdorff dimension accumulates. For example, in the discrete Jacobi setting, the Hausdorff dimension may accumulate only at one of the extrema of the spectrum \cite{Yessen2011a}. In the continuum case, it turns out, the Hausdorff dimension accumulates at infinity but, for certain values of $\lambda$, it may also accumulate at finitely many points in the spectrum. A more extensive discussion is given in Remark \ref{rem:accum}.

\item In the continuum quasicrystal setting, there is no need to restrict to locally constant potential pieces. We do this so that we can control the location of the ground state and appeal to explicit expressions for the Fricke--Vogt invariant, but it is interesting to ask whether general results of this type hold for any choice of $L^2_{\loc}$ continuum potential modelled on the Fibonacci subshift.
\end{enumerate}
\end{rem}

\section{Proof of Theorems \ref{thm:thm2}, \ref{thm:thm1}, and \ref{thm:thm3}}

Our proof relies on the dynamics of the associated renormalization map, the Fibonacci trace map. Detailed discussions on the dynamics of this map are contained in \cite{Damanik2009,Cantat2009,Casdagli1986,Roberts1996,Roberts1994b,Roberts1994} and references therein, and are complemented by \cite[Section 4.2]{Yessen2011}. We use freely standard notions from the theory of (partially) hyperbolic dynamical systems; \cite{Hasselblatt2002b,Hasselblatt2006} can serve as comprehensive references, while sufficient background is also given in \cite{Damanik2009} and the appendices of \cite{Yessen2011}. Let us briefly describe our approach.

\subsection{Background}\label{sec:background}

Given $\lambda = (\p,\q) \in \mathcal{R}$, define
\begin{align}\label{eq:line}
\ell_\lambda(E)
\eqdef\left(\frac{E-\mathfrak{q}}{2}, \frac{E}{2\mathfrak{p}}, \frac{1+\mathfrak{p}^2}{2\mathfrak{p}}\right),
\quad
E \in \R.
\end{align}
In what follows, we use the notation $\ell_\lambda$ to also denote the image of $\R$ under $\ell_\lambda$, that is, the line $\ell_\lambda(\R)$.

Given $x,y,z \in \R$, define the so-called Fricke-Vogt invariant by
\begin{align}\label{eq:invariant}
I(x,y,z) \eqdef x^2 + y^2 + z^2 - 2xyz-1,
\end{align}
and, for $V \geq 0$, consider its level surfaces
\begin{align}\label{eq:surf}
S_V\eqdef\set{(x,y,z)\in\R^3: I(x,y,z) = V}
\end{align}
(see Figure \ref{fig:fvplots} for some plots). Finally, define the \emph{Fibonacci trace map} $f: \R^3\rightarrow \R^3$ by
\begin{align*}
f(x,y,z)=(2xy-z, x, y).
\end{align*}

\begin{figure}[t!]
\subfigure[$V < 0$]{
\includegraphics[scale=0.22]{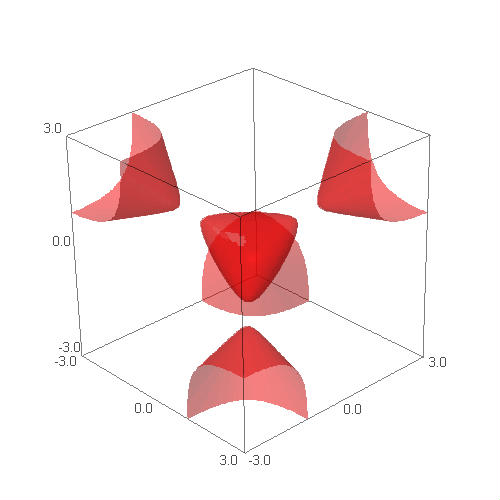}
}
\subfigure[$V = 0$]{
\includegraphics[scale=0.22]{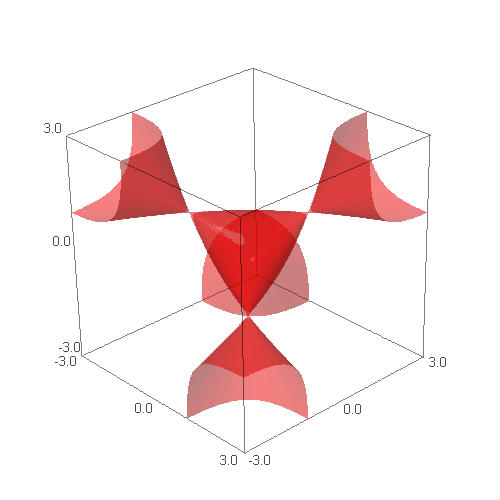}
}
\subfigure[$V > 0$]{
\includegraphics[scale=0.22]{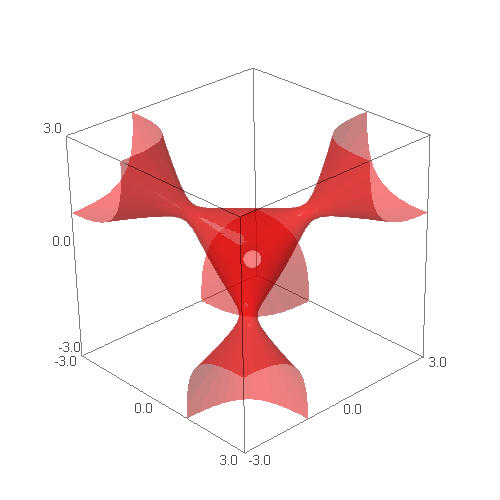}
}
\caption{Plots of the Fricke-Vogt invariant surfaces for a few values of $V$.}\label{fig:fvplots}
\end{figure}

It is easily verified that $f(S_V) = S_V$ for every $V$. Furthermore, for every $V > 0$, $f|_{S_V}$ is an Axiom A diffeomorphism \cite{Casdagli1986,Damanik2009,Cantat2009}. It is known that $E\in\R$ is in the spectrum $\Sigma_\lambda$ if and only if the forward orbit of $\ell_\lambda(E)$, $\set{f^n(\ell_\lambda(E))}_{n\in\N}$, is bounded (see the proof of Theorem 2.1 in \cite{Yessen2011a}, which is a generalization of \cite{Suto1987}). While there exist $E\in\R$ with $\ell_\lambda(E)\notin S_V$ for any $V \geq 0$, it turns out that such $E$ do not belong to $\Sigma_\lambda$ (see Section~3.2.1 in \cite{Yessen2011a}). Denoting by $\Omega_V$ the nonwandering set of $f|_{S_V}$, we have that $\bigcup_{V > 0}\Omega_V$ is partially hyperbolic (see \cite[Proposition~4.10]{Yessen2011} for the details). It is then proved that $\ell_\lambda(\Sigma_\lambda)$ is precisely the intersection set of $\ell_\lambda$ with the center-stable lamination \cite{Yessen2011a}.

This implies that $\Sigma_\lambda$ is Cantor set with quite rich multifractal structure; see \cite[Theorem~2.3]{Yessen2011a}. It is this multifractality that we exploit to prove Theorems \ref{thm:thm2}, \ref{thm:thm1}, and \ref{thm:thm3}.

Many of the results of the aforementioned works suffer from certain geometric restrictions. Namely, those results were established only for $\lambda\in\mathcal{R}$, where the line $\ell_{\lambda}$ intersects the center-stable lamination transversally. On the other hand, transversality was established for some rather restricted values of the parameters $\lambda\in\mathcal{R}$ (see \cite{Yessen2011a,Mei201x}). Recently, transversailty has been extended to all parameters of the form $(1, \mathfrak{q})$ in \cite{Damanik2014e}, and the case for $(\mathfrak{p}, 0)$ follows from \cite{Damanik2014e} without modification to the proofs (combine with the appendix in \cite{Damanik2010}). In this paper we extend transversality to \emph{all values} $(\mathfrak{p}, \mathfrak{q})\in \mathcal{R}$ as the following

\begin{thm}\label{thm:transversality}
For all $\lambda\in\mathcal{R}$ with $\lambda\neq (1, 0)$, the line $\ell_\lambda$ intersects the center-stable lamination of the Fibonacci trace map transversally.
\end{thm}

As a consequence, the said geometric restrictions can be lifted for all previous results (many of which are surveyed in \cite{Mei201x}, and some for other models appear in \cite{Damanik2013a,Damanik2013b,Yessen2011}; see also \cite{Damanik2014e}).

\begin{proof}[Proof of Theorem \ref{thm:transversality}]
It is easy to verify that for all $\mathfrak{q}\in\R$, $\ell_{(1,\mathfrak{q})}$ lies entirely in $S_{\frac{\mathfrak{q}^2}{4}}$, and it is known from \cite{Damanik2014e} that in this case $\ell_{(1, \mathfrak{q})}$ intersects the stable lamination on $S_{\frac{\mathfrak{q}^2}{4}}$ transversally. From \eqref{eq:v}, we see that for all $\mathfrak{p}\neq 0$, the line $\ell_{(\mathfrak{p},0)}$ lies on the surface $S_{V_{(\mathfrak{p},0)}}$, where $V_{(\mathfrak{p},0)} = \frac{(\mathfrak{p}-1)^2}{4\mathfrak{p}^2}$. The arguments from \cite[Section~2]{Damanik2014e} apply without modification to show that for all $\mathfrak{p}\neq 0$, $\ell_{(\mathfrak{p},0)}$ intersects the stable lamination on $S_{V_{(\mathfrak{p},0)}}$ transversally. On the other hand, it is known from \cite[Proposition~4.10]{Yessen2011} that the center-stable manifolds intersect the surfaces $\set{S_V}_{V > 0}$ transversally (the intersection of the center-stable manifolds with $S_V$ gives the stable lamination on $S_V$). Since the points in the intersection of $\ell_\lambda$ with the center-stable manifolds
%
%
correspond to the points in the spectrum, and since the spectrum is compact, it follows by continuity that for each fixed $\mathfrak{p}_0\notin \set{0,1}$, there exists $\delta > 0$ such that for all $\mathfrak{q}\in(-\delta, \delta)$, $\ell_{(\mathfrak{p}_0, \mathfrak{q})}$ intersects the center-stable manifolds transversally; similarly, for each fixed $\mathfrak{q}_0 \in \R$, there exists $0 < \delta < 1$ such that for all $\mathfrak{p}\in (1 - \delta, 1 + \delta)$, $\ell_{(\mathfrak{p}, \mathfrak{q}_0)}$ intersects the center-stable manifolds transversally (compare with \cite[Theorem~2.5]{Yessen2011a}).

Let us argue by contradiction that tangencies cannot occur. Fix $\mathfrak{p}_0\notin\set{0,1}$. Let us assume that $\mathfrak{q}_0 > 0$ (respectively, $\mathfrak{q}_0 < 0$) is such that for all $\mathfrak{q}\in[0, \mathfrak{q}_0)$ (respectively, $\mathfrak{q}\in (\mathfrak{q}_0, 0]$), $\ell_{(\mathfrak{p}_0, \mathfrak{q})}$ intersects the center-stable manifolds transversally, while $\ell_{(\mathfrak{p}_0,\mathfrak{q}_0)}$ is tangent to some center-stable manifold. In what follows, let us assume without loss of generality, that $\mathfrak{q}_0 > 0$.

Denote by $p$ a point of tangency between $\ell_{(\mathfrak{p}_0, \mathfrak{q}_0)}$ and a center-stable manifold. We claim that $p\notin S_0$. Indeed, it is easy to see from \eqref{eq:v} that for all $\lambda\in\mathcal{R}$ with $\lambda\neq (1,\mathfrak{q})$ and $\lambda\neq (\mathfrak{p},0)$, $\ell_\lambda$ intersects $S_V$ transversally for every $V$ (differentiate \eqref{eq:v} with respect to $E$ and notice that the result is zero if and only if $\mathfrak{q} = 0$ or $\mathfrak{p}=1$, independently of $E$). On the other hand, if the intersection of $\ell_\lambda$ and the center-stable lamination contains a point $p_0 \in S_0$, then $p_0$ must lie on the strong stable manifold of one of the singularities of $S_0$ (see \cite[Lemma~2.2]{Damanik2013a}, and use equations \eqref{eq:line} and \eqref{eq:invariant} to see that $p_0 \notin [-1, 1]^3 \subset \R^3$); on the other hand, those center-stable manifolds that intersect $S_0$ along the aforementioned strong-stable manifolds are tangent to $S_0$ (and this tangency is quadratic; see \cite[Lemma~4.8]{Mei201x} for the details); combined with the fact that $\ell_\lambda$ is transversal to $S_0$, we see that $\ell_{\lambda}$ is transversal to the center-stable manifold at $p_0$.

Now, we finish the proof by ruling out the possibility $p\notin S_0$.

It is known that for all $\lambda\in\mathcal{R}$ with $\lambda\neq (1,0)$, tangential intersections of $\ell_\lambda$ with the center-stable manifolds away from $S_0$ are isolated, if they exist (see the proof of \cite[Theorem~2.3(i)]{Yessen2011a}). Let $U$ be a small open neighborhood of $p$ such that $p$ is the only tangency in $U$, and $U\cap S_0 = \emptyset$.

It can be shown that the center-stable manifolds are analytic away from $S_0$ (see \cite[Section~2]{Buzzard2001} for a discussion). Let $W(p)$ denote the (part of the) center-stable manifold in $U$ containing $p$. Perform an analytic change of coordinates to map $p$ to the origin in $\R^3$, and $W(p)\cap U$ to a part of the $xy$ plane. Let us call this change of coordinates $\Phi$. Then $g: E\mapsto \pi_z\circ\Phi\circ \ell_{(\mathfrak{p_0}, \mathfrak{q})}(E)$, $\mathfrak{q}\in [0, \mathfrak{q}_0]$ where $\pi_z$ is the projection onto the $z$ coordinate, is analytic and depends analytically on $\mathfrak{q}$. Moreover, if $E_p$ is such that $g(E_p) = \Phi(p)$, then $g'(E_p) = 0$ and there exists $\epsilon > 0$ such that for all $E \in (E_p - \epsilon, E_p + \epsilon)$, with $E\neq E_p$, $g'(E)\neq 0$. The arguments from \cite[Section 2]{Damanik2014e} apply without modification to guarantee that either the tangency at $p$ is quadratic (i.e. $E_p$ is a root of $g$ of order $2$), or that $E_p$ is a root of $g$ of order $k > 2$ and for all $\mathfrak{q}\in[0, q_0)$, $g$ has precisely $k$ roots $E_1, \dots, E_k$ with $g'(E_j)\neq 0$ for all $j = 1, \dots, k$. Moreover, these roots approach the origin as $\mathfrak{q}$ approaches $\mathfrak{q}_0$. Combined with the fact that the center-stable manifolds vary continuously in the $C^2$ topology (see \cite[Section B.1.2]{Yessen2011}), the arguments from \cite[Section 2]{Damanik2014e} again apply without modification to guarantee a tangency of $\ell_{(\mathfrak{p}_0, \mathfrak{q})}$ with a center-stable manifold in $U$ for some $\mathfrak{q}\in (0, \mathfrak{q}_0)$, which contradicts the assumption that no tangencies occur for all $\mathfrak{q}\in [0, \mathfrak{q}_0)$. Thus the tangency at $p$ must be quadratic. On the other hand, the arguments in \cite[Section 2]{Damanik2014e} guarantee that the tangency at $p$ cannot be quadratic, as this would lead to either an isolated point in the spectrum (which is precluded by the fact that the spectrum is known to be a Cantor set; see \cite{Yessen2011a}), or to closure of a gap in the spectrum, which is precluded by Claim \ref{claim:trans} below (which states that no two center-stable manifolds that mark the endpoints of the same gap can intersect away from $S_0$). Thus $p$ cannot be a point of tangency. This completes the proof; let us only remark that in the application of \cite[Section 2]{Damanik2014e}, one needs to know that all the intersections between the complexified $\ell_\lambda$ and the complexified center-stable manifolds are real. This follows, just like in \cite{Damanik2014e}, from an application of the complexified version of S\"ut\H{o}'s arguments (see \cite{Yessen2011a}) which guarantees that if $E \in \C$ is such that $\ell_\lambda(E)$ has a bounded forward orbit (and it does provided that $\ell_\lambda(E)$ is a point on a (complexified) center-stable manifold), then $E$ is an element of the spectrum; on the other hand, since $H_\lambda$ is self-adjoint, its spectrum is real.

\begin{claim}\label{claim:trans}
For every gap $G$ of the Cantor set given by the intersection of $\ell_{(\mathfrak{p}_0, 0)}$ with the center-stable manifolds inside $U$, there exist center-stable manifolds $W_1$ and $W_2$ such that the endpoints of $G$ are given by the intersection of $\ell_{(\mathfrak{p}_0, 0)}$ with $W_1$ and $W_2$, and $W_1\cap W_2 = \emptyset$ away from the surface $S_0$.

\end{claim}

\begin{proof}
This follows from \cite[Theorem 1.3]{Damanik2014e} (the theorem is stated for the Fibonacci Hamiltonian of the form $H_{(1, \mathfrak{q})}$, but the proof applies to the case $(\mathfrak{p}_0, 0)$ without modification).
\end{proof}
\end{proof}

The notions of thickness of a Cantor set, as well as the notions of (local) Hausdorff and (upper) box-counting dimensions will play a crucial role in the rest of the paper; these notions are discussed in detail in \cite[Chapter~4]{Palis1993}, for example. The following notation will be used throughout the remainder of the paper.

\begin{itemize}

\item The Hausdorff dimension of a set $A$: $\hdim(A)$.
\item The box-counting dimension of a set $A$: $\bdim(A)$; the upper box-counting dimension of $A$ will be denoted by $\overline{\bdim}(A)$.
\item The thickness of a set $A$: $\tau(A)$.
\item For $\star \in \set{\hdim, \tau}$, the local $\star$ of a set $A$ at $a\in A$ is denoted by $\star^{\loc}(A, a)$.

\end{itemize}

We are now ready to prove Theorems \ref{thm:thm2}, \ref{thm:thm1}, and \ref{thm:thm3}.

\subsection{Proof of Theorem \ref{thm:thm2}}

Given $\lambda = (\mathfrak{p}, \mathfrak{q}) \in \mathcal{R}$, denote by $\Pi_{\lambda}$ the plane $\set{(x,y,z) \in \R^3 : z = \frac{1+\mathfrak{p}^2}{2\mathfrak{p}}}$; note that $\Pi_\lambda$ only depends on $\mathfrak{p}$. Clearly, $\ell_{\lambda}$ is contained in $\Pi_{\lambda}$ for all $\lambda\in\mathcal{R}$. Set $y = y(E) = \frac{E}{2\mathfrak{p}}$ and $x = x(E) = \frac{E - \mathfrak{q}}{2}$. Then we have
\begin{align}\label{eq:app-1}
y = y(x) \eqdef \frac{1}{\mathfrak{p}}x + \frac{\mathfrak{q}}{2\mathfrak{p}},
\end{align}
and
\begin{align*}
\ell_{(\mathfrak{p},\mathfrak{q})}(E) = \left(x, y(x), \frac{1+\mathfrak{p}^2}{2\mathfrak{p}}\right).
\end{align*}

Thus, for any fixed $\lambda\in\mathcal{R}$, $\ell_{\lambda}$ can be viewed as a line in $\Pi_{\lambda}$ with parameterization $(x, y(x))$, $x\in\R$.

\begin{lem}\label{lem:app-1}
For all $\Delta > 0$ there exists $\mathfrak{p}$ with $\abs{\mathfrak{p}} > \Delta$ and $\mathfrak{q}$ depending on $\mathfrak{p}$, such that there exists $E\in\Sigma_{(\mathfrak{p},\mathfrak{q})}$ with $\ell_{(\mathfrak{p},\mathfrak{q})}(E)\in S_0$.
\end{lem}

\begin{proof}
Let us denote the point $(1,1,1) \in S_0$ by $P_1$; $P_1$ is a partially hyperbolic fixed point with one-dimensional stable, unstable, and center directions ($Df$ acts as an isometry on the center subspace). Obviously, since the surfaces $S_V$ are invariant, the strong-stable manifold of $P_1$ lies on $S_0$. Notice that $P_1$ is a cut point of its strong-stable manifold; after removing $P_1$ from the strong-stable manifold, we end up with one of the two branches, which we denote by $W$, which satisfies the following. For all $p\in W$, $\abs{f^{-n}(p)}\underset{n\rightarrow\infty}{\longrightarrow}\-\infty$ (see \cite[Section~5]{Cantat2009}; in the terminology of \cite{Cantat2009}, $P_1$ is an \emph{s-one-sided} or \emph{stably one-sided point}). Consequently, every $p\in W$ escapes to infinity in every coordinate under $f^{-1}$ (see Section 3 in \cite{Roberts1996}). Since $W = f^{-1}(W)$, it follows that there exist points of $W$ with arbitrarily large $z$ coordinate in absolute value. It follows by continuity (indeed, $W$ is smooth), that for any arbitrarily large $\Delta > 0$ there exists $\mathfrak{p}$ with $\abs{\mathfrak{p}}>\Delta$ such that $\Pi_{(\mathfrak{p},\mathfrak{q})}$ intersects $W$; pick such a $\mathfrak{p}$. Now from \eqref{eq:app-1}, we see that for any such choice of $\mathfrak{p}$ we can find $\mathfrak{q}$ such that $\ell_{(\mathfrak{p},\mathfrak{q})}$ intersects $W$. Indeed, $\Pi_\lambda$ is determined only by $\mathfrak{p}$, and the slope of the line $\ell_\lambda$ when viewed as a line in $\Pi_\lambda$ in terms of \eqref{eq:app-1} is also determined only by $\mathfrak{p}$. It follows that for any point $p$ in $\Pi_\lambda$, we can choose $\mathfrak{q}$ such that the line in \eqref{eq:app-1} passes through $p$.
\end{proof}

\begin{lem}\label{lem:app-3}
For all $\Delta > 1$ and $\tau_0 > 0$ there exists $(\mathfrak{p},\mathfrak{q})\in\mathcal{R}$ with $\abs{\mathfrak{p}}> \Delta$, and $\epsilon > 0$, such that for all $\lambda\in\mathcal{R}$ with $\norm{(\mathfrak{p},\mathfrak{q})-\lambda}< \epsilon$, there exists $E\in\Sigma_{\lambda}$ with $\tau^\loc(\Sigma_{\lambda}, E) > \tau_0$.
\end{lem}

\begin{proof}
It is known that the center-stable manifold that contains $P_1$ is tangent to $S_0$; its intersection with $S_0$ is precisely the strong-stable manifold of $P_1$ (see \cite[Lemma~4.8]{Mei201x}). Let us denote this center-stable manifold by $W^{cs}(P_1)$.

Now pick $\mathfrak{p}_0$ with $\abs{\mathfrak{p}_0} > \Delta$ and a suitable $\mathfrak{q}_0$ (guaranteed by Lemma \ref{lem:app-1}), such that the corresponding $\ell_{(\mathfrak{p}_0, \mathfrak{q}_0)}$ intersects $W^{cs}(P_1)$ at some point along $W$. Notice that $\mathfrak{q}_0\neq 0$, since $\ell_{(\mathfrak{p}_0,0)}$ lies entirely in the surface $S_V$ with $V = \frac{(\mathfrak{p}_0^2-1)^2}{4\mathfrak{p}_0^2}$ (see equation \eqref{eq:v} below). Observe that for all $\delta > 0$ or $\delta < 0$ with $\abs{\delta}$ sufficiently small, $\ell_{(\mathfrak{p}_0, \mathfrak{q}_0+\delta)}$ still intersects $W^{cs}(P_1)$ in some point $\ell_{(\mathfrak{p}_0, \mathfrak{q}_0+\delta)}(E_\delta)$, $E_\delta \in \Sigma_{(\mathfrak{p}_0, \mathfrak{q}_0+\delta)}$, but this intersection no longer occurs on $S_0$, but on some $S_{V_\delta}$ with $V_\delta > 0$ and $V_\delta \rightarrow 0$ as $\delta\rightarrow 0$. Indeed, since $W$ is analytic, if for infinitely many small $\delta > 0$, $\ell_{(\mathfrak{p}_0, \mathfrak{q}_0 + \delta)}(E_\delta)\in S_0$, then we must have $W$ lying entirely in $\Pi_{(\mathfrak{p}_0, \mathfrak{q}_0 + \delta)}$ (which, recall, depends only on $\mathfrak{p}_0$), which is impossible since $\Pi_{(\mathfrak{p}_0, \mathfrak{q}_0 + \delta)}$ does not contain $P_1$, while $W$ contains points arbitrarily close to $P_1$.

In what follows, let us assume, without loss of generality, that $\delta > 0$. Moreover, let us remark, for future reference, that
\begin{align}\label{eq:text}
V_\delta
=
\min \set{ I ( \ell_{ ( \p_0,\q_0 + \delta ) }(E) ) : E \in \Sigma_{ ( \p_0, \q_0 + \delta ) } }.
\end{align}
Indeed, this follows by monotonicity (see \eqref{eq:v} below).

By continuity, the intersection of $\ell_{(\mathfrak{p}_0, \mathfrak{q}_0+\delta)}$ with $W^{cs}(P_1)$ at $\ell_{(\mathfrak{p}_0, \mathfrak{q}_0+\delta)}(E_\delta)$ is transversal uniformly in $\delta$ for all sufficiently small $\delta >0$. Thus, in a sufficiently small neighborhood of $\ell_{(\mathfrak{p}_0, \mathfrak{q}_0+\delta)}(E_\delta)$, $\ell_{(\mathfrak{p}_0, \mathfrak{q}_0+\delta)}$ intersects the center-stable manifolds uniformly transversally (since the center-stable manifolds form a continuous family in the $C^1$ topology away from $S_0$; see \cite[Proposition~4.10]{Yessen2011} for the details). Let us take such a neighborhood of $\ell_{(\mathfrak{p}_0, \mathfrak{q}_0+\delta)}(E_\delta)$ and denote it by $\mathcal{V}_\delta$. Pick a plane $\Lambda$ containing $\ell_{(\mathfrak{p}_0, \mathfrak{q}_0+\delta)}$ and transversal to the surfaces $S_V$, $V \geq 0$, as well as to the center-stable manifolds, inside the neighborhood $\mathcal{V}_\delta$. Observe that the intersection of $\Lambda\cap\mathcal{V}_\delta$ with the surfaces $\set{S_V}_{V\geq 0}$ produces a smooth foliation of $\Lambda\cap\mathcal{V}_\delta$. We shall denote this foliation by $\mathcal{S}$. Let us further assume that this foliation has been rectified (notice that after the rectification, $\ell_{(\mathfrak{p}_0, \mathfrak{q}_0+\delta)}$ is in general no longer a line, but we abuse the notation and use the same symbol $\ell_{(\mathfrak{p}_0, \mathfrak{q}_0+\delta)}$). The intersection of the center-stable manifolds with $\Lambda\cap\mathcal{V}_\delta$ produces a lamination with smooth leaves. We shall call this lamination $\mathcal{L}$.

Let us parameterize the leaves of $\mathcal{S}$ by $V$ (same as for $S_V$), and call the leaves $L_V$. We know that for every $V$, the intersection of $L_V$ with the leaves of the lamination $\mathcal{L}$ is a Cantor set.

Pick an arbitrary but small $\epsilon > 0$, such that if $J_{\delta, \epsilon}$ is a compact interval of length $\epsilon$ along $\ell_{(\mathfrak{p}_0, \mathfrak{q}_0+\delta)}$ having $\ell_{(\mathfrak{p}_0, \mathfrak{q}_0+\delta)}(E_\delta)$ as one of its endpoints, $J_{\delta, \epsilon}\subset \mathcal{V}_\delta$ and $J_{\delta, \epsilon}$ intersects the leaves of $\mathcal{L}$. Denote $\mathcal{C}_{\delta, \epsilon}\eqdef J_{\delta,\epsilon}\cap \mathcal{L}$, which, by adjusting $\epsilon$, can be ensured to be a Cantor set (indeed, for any given $\epsilon$, $\mathcal{C}_{\delta, \epsilon}$ is either a Cantor set or a Cantor set together with an isolated point on the boundary of $J_{\delta, \epsilon}$). Assume also without loss of generality that $J_{\delta,\epsilon}$ is the convex hull of $C_{\delta, \epsilon}$. For what follows, refer to Figure \ref{fig:guide} for visual guidance.

\begin{centering}
\begin{figure}[t]
\includegraphics[scale=.5]{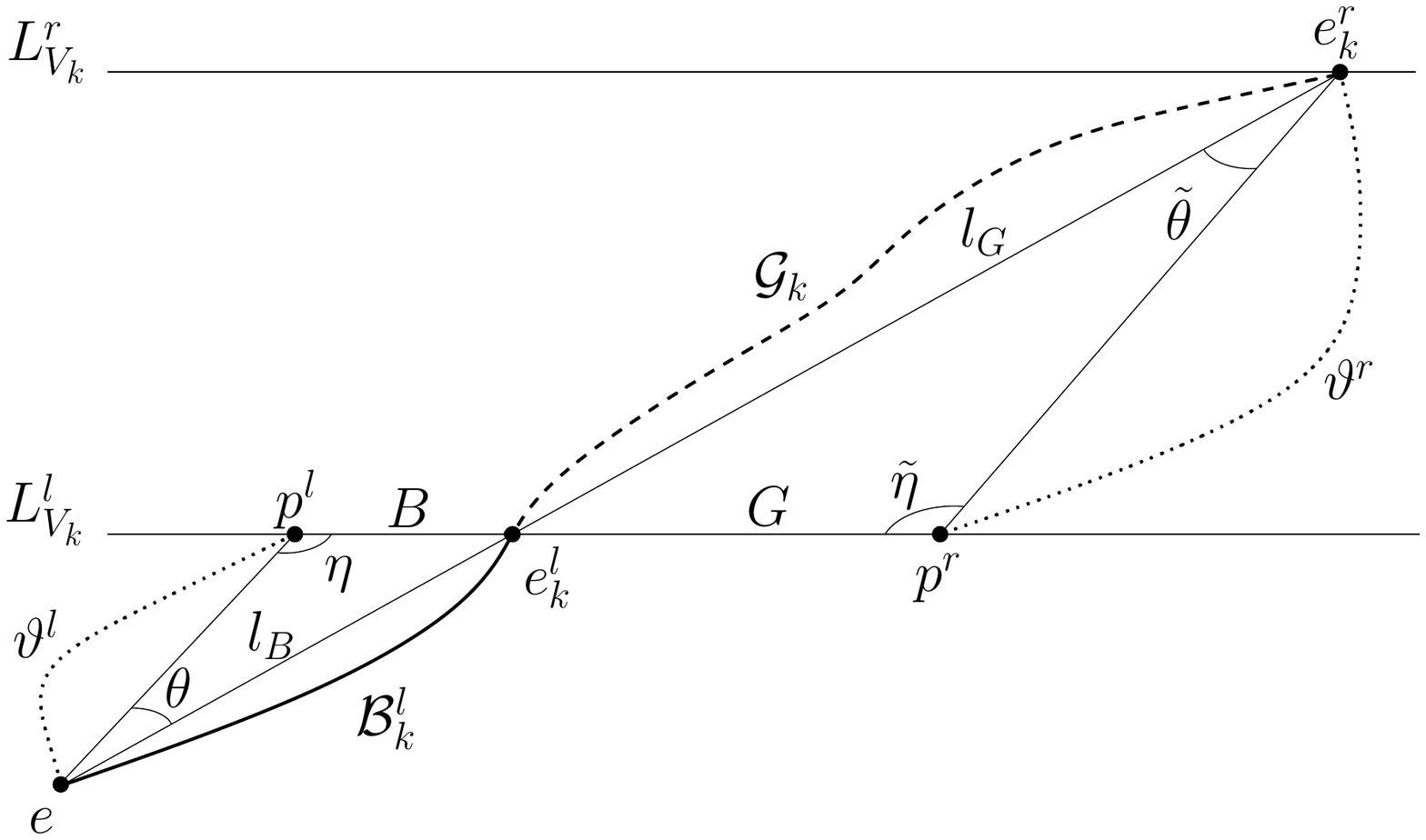}
\caption{}
\label{fig:guide}
\end{figure}
\end{centering}

Let $\set{\mathcal{G}_i}$ be the set of ordered gaps of the Cantor set $\mathcal{C}_{\delta, \epsilon}$. Let $\set{\mathcal{B}_i^l, \mathcal{B}_i^r}$ be the set of the corresponding bands (that is, $\mathcal{B}_k^l$ and $\mathcal{B}_k^r$ are the subintervals of $J_{\delta, \epsilon}\setminus\bigcup_{i=1}^k\mathcal{G}_i$ immediately to the left and immediately to the right of $\mathcal{G}_k$, respectively). For any $k$, denote the two endpoints of $\mathcal{G}_k$ by $e_k^l$ and $e_k^r$, with the assumption that $e_k^l\in\mathcal{B}_k^l$ and $e_k^r\in\mathcal{B}_k^r$. Let us write also $L_{V_k}^\bullet$, $\bullet\in\set{l, r}$, for the leaf of $\mathcal{S}$ containing $e_k^\bullet$. Identify by $p^r$ the point $\vartheta^r\cap L_{V_k}^l$, where $\vartheta^r$ is the leaf of $\mathcal{L}$ whose intersection with $\ell_{(\mathfrak{p}_0, \mathfrak{q}_0+\delta)}$ gives $e_k^r$. If $\vartheta^l\in\mathcal{L}$ is the leaf whose intersection with $\ell_{(\mathfrak{p}_0, \mathfrak{q}_0+\delta)}$ gives the other endpoint of $\mathcal{B}_k^l$, say $e$, denote by $p^l$ the point $\vartheta^l\cap L_{V_k}^l$, and denote by $l_B$ the line segment connecting $e$ and $e_k^l$. Denote by $B$ the line segment connecting $p^l$ and $e_k^l$. Denote the line segment connecting $e_k^l$ and $e_k^r$ by $l_G$. There exist constants $C_1, C_2 > 0$ independent of $\delta$ and $\epsilon$ such that
\begin{align}\label{eq:affinity}
C_1\abs{l_G}\geq \abs{\mathcal{G}_k}_{\ell}\hspace{2mm}\text{ and }\hspace{2mm}
C_2\abs{l_B}\leq \abs{\mathcal{B}_k^l}_{\ell}
\end{align}
(here $\abs{\cdot}$ denotes the length of the line segment, and $\abs{\cdot}_\ell$ denotes the length along the curve $\ell_{(\mathfrak{p}_0, \mathfrak{q}_0+\delta)}$); see \cite[Lemma 3.3]{Yessen2011a} for a simple justification of this.

Now (refer to Figure \ref{fig:guide}) observe that since $\vartheta^{l,r}$ are transversal to the leaves of $\mathcal{S}$ as well as $\ell_{(\mathfrak{p}_0, \mathfrak{q}_0+\delta)}$ (albeit the angle of intersection may depend on $\delta$), assuming that $\epsilon$ was initially chosen sufficiently small (depending on $\delta$ only) we can guarantee that the angles $\theta$, $\eta$, $\tilde{\theta}$ and $\tilde{\eta}$ are such that
\begin{align*}
\frac{\sin(\eta)\sin(\tilde{\theta})}{\sin(\tilde{\eta})\sin(\theta)} > \frac{1}{2}.
\end{align*}
Combining this with \eqref{eq:affinity}, we obtain
\begin{align*}
\frac{\abs{\mathcal{B}_k^l}_\ell}{\abs{\mathcal{G}_k}_\ell}\geq \frac{C_2\abs{l_B}}{C_1\abs{l_G}} = \frac{\abs{B}}{\abs{G}}\frac{C_2\sin(\eta)\sin(\tilde{\theta})}{C_1\sin(\tilde{\eta})\sin(\theta)} > \frac{C_2}{2C_1}\frac{\abs{B}}{\abs{G}}.
\end{align*}
Similar bounds are obtained for the quotient $\abs{\mathcal{B}_k^r}_\ell/\abs{\mathcal{G}_k}_\ell$. On the other hand, since $C_2$ and $C_1$ are universal constants, we can guarantee that
\begin{align*}
\frac{C_2}{2C_1}\frac{\abs{B}}{\abs{G}} > \tau_0
\end{align*}
by choosing $\delta$ suitably small, and choosing $\epsilon$ (depending on $\delta$) sufficiently small so that all the bounds above hold; indeed, this follows since the thickness of the Cantor sets obtained by intersecting the leaves of $\mathcal{S}$ with those of $\mathcal{L}$ tends to infinity at $S_0$ (see \cite{Damanik2010}). Since these estimates are independent of the gap index $k$, we conclude that, for suitably small $\delta$ and $\epsilon$, the thickness of $\mathcal{C}_{\delta, \epsilon}$ is larger than $\tau_0$. This also holds for all parameters $(\mathfrak{p},\mathfrak{q})$ sufficiently close to $(\mathfrak{p}_0, \mathfrak{q}_0+\delta)$.
\end{proof}

\begin{lem}\label{lem:app-4}
For every $\Delta, \delta > 0$, there exists a nonempty open set $\mathcal{U} \subset \mathcal{R}$, such that the following hold for all $\lambda\in\mathcal{U}$.

\begin{enumerate}\itemsep0.5em

\item[\textup{(1)}] There exists $E_0\in\Sigma_\lambda$ such that $\tau^\loc(\Sigma_\lambda, E_0) > \Delta$.

\item[\textup{(2)}] At one of the extrema of $\Sigma_\lambda$, $E_\mathrm{end}$, we have $\lhdim(\Sigma_\lambda, E_\mathrm{end}) < \delta$.
\end{enumerate}
\end{lem}

\begin{proof}
Recall that the spectral radius of a self-adjoint operator is equal to its norm, and that $\norm{H_{(\mathfrak{p}, \mathfrak{q})}}$ is of order $\max\set{\abs{\mathfrak{p}}, \abs{\mathfrak{q}}}$. Consequently, for all $C > 0$, $\Delta > 1$, and $\tau_0 > 0$, we can always find $\lambda = (\mathfrak{p}$, $\mathfrak{q}) \in \mathcal{R}$, and $E_0\in \Sigma_\lambda$ so that the conclusion of Lemma~\ref{lem:app-3} holds and such that $I(\ell_\lambda(E_1)) > C$ for some $E_1\in \Sigma_\lambda$. Indeed, this follows since
\begin{align}\label{eq:v}
V_\lambda(E)
\eqdef
I(\ell_\lambda(E))
=
\frac{\mathfrak{q}(\mathfrak{p}^2-1)E + \mathfrak{q}^2 + (\mathfrak{p}^2 - 1)^2}{4\mathfrak{p}^2}
\end{align}
for all $E$, which can easily by computed from \eqref{eq:surf}. Evidently, $V_\lambda$ is a monotone function of $E$, so we can take $E_1$ to be one of the extrema of $\Sigma_\lambda$.

On the other hand, for any $E\in\Sigma_\lambda$,
\begin{align*}
\hdim^\loc(\Sigma_\lambda, E) = \frac{1}{2}\hdim(\Omega_{V_\lambda(E)}),
\end{align*}
where $\Omega_V$ is the nonwandering set of $f|_{S_V}$ (see the proof of in \cite[Theorem~2.1(iii)]{Yessen2011}). Moreover, from \cite{Damanik2008}, we have
\begin{align*}
\lim_{V\rightarrow\infty}\hdim(\Omega_V)=0,
\end{align*}
so we can get (1) and (2) for $\lambda$ by taking $C > 0$ sufficiently large above. That small perturbations of $\lambda$ do not destroy these bounds follows from Lemma \ref{lem:app-3} and the fact that $\hdim(\Omega_V)$ is continuous in $V$.
\end{proof}

\begin{rem}\label{rem:extremum}
It is not difficult to see from the proof of Lemma \ref{lem:app-3} that $E_0$ can be taken arbitrarily close to the other extremum of the spectrum. In fact, $\mathcal{U}$ can be adjusted so that $E_0$ can be taken to be an extremum of $\Sigma_\lambda$ (that is, we need to make sure that $\ell_\lambda(E_0)$ does not lie on $S_0$ in the proof of Lemma \ref{lem:app-3}).
\end{rem}

\begin{proof}[Proof of Theorem~\ref{thm:thm2}]
Take $\Delta > 1$ and $\delta\in(0, 1/2)$. With this choice of $\Delta$ and $\delta$, let $\mathcal{U}$ be as in Lemma \ref{lem:app-4}. For $\lambda_i\in \mathcal{U}$ with $i = 1, 2$, let $E_0(\lambda_i)$ and $E_\mathrm{end}(\lambda_i)$ be as in the Lemma.

Let $E_\mathrm{end}^2 = E_\mathrm{end}(\lambda_1)+E_\mathrm{end}(\lambda_2)$. Notice that $E_\mathrm{end}^2$ is an extremum of $\Sigma_{(\lambda_1, \lambda_2)}^2$; let us assume, without loss of generality, that it is the supremum (so $E_\mathrm{end}(\lambda_1)$ and $E_\mathrm{end}(\lambda_2)$ are also the suprema, provided that $\mathcal{U}$ is chosen sufficiently small in the beginning). Take $\epsilon > 0$ such that
\begin{align*}
\hdim([E_\mathrm{end}(\lambda_i)-\epsilon, E_\mathrm{end}(\lambda_i)]\cap \Sigma_{\lambda_i}) < \delta.
\end{align*}
We claim that there exists $\tilde{\epsilon} > 0$ such that
\begin{align*}
\hdim([E_\mathrm{end}^2 - \tilde{\epsilon}, E_\mathrm{end}^2]\cap \Sigma_{(\lambda_1, \lambda_2)}^2) < 1.
\end{align*}
Indeed, for arbitrary sets $A, B\subset\R$, we have
\begin{align}\label{eq:dim-sum}
\hdim(A + B) \leq \min\set{\overline{\bdim}(A)+\hdim(B), 1}
\end{align}
(see \cite[Theorem~8.10(2)]{Mattila1995}). On the other hand, we have
\begin{align}\label{eq:haus-box}
\begin{split}
&\hdim([E_\mathrm{end}(\lambda_i)-\epsilon, E_\mathrm{end}(\lambda_i)]\cap \Sigma_{\lambda_i}) \\
=\hspace{2mm}&\overline{\bdim}([E_\mathrm{end}(\lambda_i)-\epsilon, E_\mathrm{end}(\lambda_i)]\cap \Sigma_{\lambda_i}),
\end{split}
\end{align}
which follows from the transversality of intersection of $\ell_\lambda$ with the center-stable manifolds (see \cite[Proposition~3.2]{Yessen2011a}). Thus $[E_\mathrm{end}^2-\tilde{\epsilon}, E_\mathrm{end}^2]\cap \Sigma_{(\lambda_1, \lambda_2)}^2$ has Hausdorff dimension strictly smaller than one, hence is of zero Lebesgue measure. Since it is also compact and does not contain any isolated points (except possibly the boundary $E_\mathrm{end}^2 - \tilde{\epsilon}$, which can be remedied by decreasing $\tilde{\epsilon}$), it follows that it is a Cantor set.

Finally, given that $\Delta > 1$, it follows from the Gap Lemma of S.\ Newhouse (see, e.g., \cite[Theorem~2.2]{Astels2000} for a derivation) that $\Sigma_{(\lambda_1, \lambda_2)}^2$ contains an interval around $E_0(\lambda_1)+E_0(\lambda_2)$.
\end{proof}

\begin{rem}\label{rem:thm2-details}
As mentioned in Remark \ref{rem:thm1}, we suspect that for all $\lambda_1, \lambda_2\in\mathcal{R}$ sufficiently close to $(1, 0)$, $\Sigma_{(\lambda_1, \lambda_2)}^2$ is an interval. A proof of this would involve control of the global thickness of $\Sigma_\lambda$ for $\lambda$ close to $(1, 0)$. Such control can be obtained via finer estimates than those used in the proof above, but the arguments would be far more technical. We have decided to delegate this task to later investigations.
\end{rem}

\subsection{Proof of Theorem \ref{thm:thm1}}\label{sec:proof-thm-2}

In what follows, we need a slight generalization of \cite[Proposition~2.3]{Damanik2013X}, namely Proposition~\ref{prop:key_prop} below. For the sake of convenience and completeness, let us first recall the general setup of \cite{Damanik2013X}.

Let $\mathcal{B}$ be a finite set of cardinality $|\mathcal{B}| \geq 2$, and consider $\mathcal{B}^{\mathbb{Z}_{+}}$, the standard symbolic space, equipped with the product topology. Here $\Z_+ = \set{0, 1, 2, \dots}$.
Let $\mathcal{E}$ be a Borel subset of $\mathcal{B}^{\mathbb{Z}_{+}}$, and
let $\mu$ be a probability measure on $\mathcal{E}$.

Let $J \subset \mathbb{R}$ be a compact nondegenerate interval. We assume that we are given a
family of continuous maps
\begin{equation*}
\Pi_{j} : \mathcal{E} \to \mathbb{R}, \ j \in J,
\end{equation*}
and
\begin{equation*}
\nu_{j} = \Pi_{j}(\mu) \eqdef \mu \circ \Pi_{j}^{-1}.
\end{equation*}

For a word $u \in \mathcal{B}^{n}, n \geq 0$, we denote by $|u| = n$ its length and by $[u]$ the cylinder set of
elements of $\mathcal{E}$ that have $u$ as a prefix. More precisely,
$[u] = \left\{ \omega \in \mathcal{E} : \omega_{0} \cdots \omega_{n-1} = u \right\}$.
For $\omega, \tau \in \mathcal{E}$, we write $\omega \wedge \tau$ for the maximal common prefix of
$\omega$ and $\tau$, which is empty if $\omega_{0} \neq \tau_{0}$; we set the length of the empty word to be zero.
Furthermore, for $\omega, \tau \in \mathcal{E}$, let
\begin{equation*}
\phi_{\omega, \tau}(j) \eqdef \Pi_{j}(\omega) - \Pi_{j}(\tau).
\end{equation*}
We write $\mathcal{L}^{1}$ for one-dimensional Lebesgue measure on $\R$.

\begin{prop}\label{prop:key_prop}
Let $\eta$ be a compactly supported Borel measure on the real line.
Suppose that the following holds. For every $\epsilon > 0$ there exists $\mathcal{E}_0\subset \mathcal{E}$ with $\mu(\mathcal{E}_0) > 1 - \epsilon$ and constants $C_{1}, C_{2}, C_{3}, \alpha, \beta, \gamma > 0$ and $k_{0} \in \mathbb{Z}_{+}$ such that
\begin{equation}\label{dimension}
\lhdim(\eta, x)\eqdef\lim_{r \downarrow 0} \frac{ \log \eta(B_{r}(x)) }{ \log r } \geq d_{\eta}  \ \text{ for $\eta$-a.e. $x$, }
\end{equation}
where $B_{r}(x) = [x - r, x + r]$ and $d_\eta$ satisfies
\begin{equation}
d_{\eta} + \frac{\gamma}{\beta} > 1 \ \text{ and } \ d_{\eta} > \frac{\beta - \gamma}{\alpha};
\end{equation}
\begin{equation}
\max_{j \in J} |\phi_{\omega, \tau}(j)| \leq C_{1} \abs{\mathcal{B}}^{-\alpha |\omega \wedge \tau| }
\ \text{ for all } \omega, \tau \in \mathcal{E}_{0}, \ \omega \neq \tau;
\end{equation}
\begin{equation}
\sup_{v \in \mathbb{R}} \mathcal{L}^{1}
\left( \left\{  j \in J : |v + \phi_{\omega, \tau} (j) | \leq r \right\} \right)
\leq C_{2} \abs{\mathcal{B}}^{|\omega \wedge \tau| \beta} r \ \text{ for all } \omega, \tau \in \mathcal{E}_{0}, \ \omega \neq \tau
\end{equation}
such that $|\omega \wedge \tau| \geq k_{0}$,  and
\begin{equation}
\max_{u \in \mathcal{B}^{n}, [u] \cap \mathcal{E}_{0} \neq \emptyset} \mu([u]) \leq C_{3} \abs{\mathcal{B}}^{-\gamma n} \
\text{ for all } n \geq 1.
\end{equation}
Then, $\eta \ast \nu_{j} \ll \mathcal{L}^{1}$ for Lebesgue-a.e. $j \in J$.
\end{prop}

The only difference between Proposition \ref{prop:key_prop} and \cite[Proposition~2.3]{Damanik2013X} is condition (\ref{dimension}). In \cite{Damanik2013X}, the measure $\eta$ is exact dimensional. Nevertheless, Proposition \ref{prop:key_prop} can be proved by repeating verbatim the proof of \cite[Proposition 2.3]{Damanik2013X}.

\begin{figure}[t]
\includegraphics{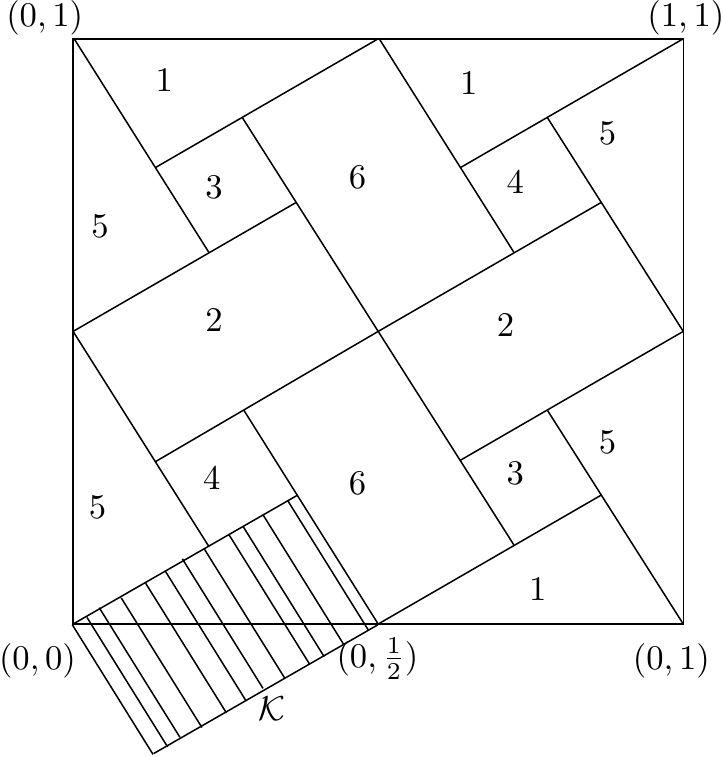}
\caption{The Markov Partition for the map $\mathcal{A}$, the rectangle $\mathcal{K}$, and the stable manifolds
which lie inside $\mathcal{K}$.}\label{fig:markov}
\end{figure}

Let
\begin{equation}\label{eq:cc-center}
\mathbb{S} \eqdef S_{0} \cap \left\{  (x, y, z) \in \mathbb{R}^3 : |x| \leq 1, |y| \leq 1, |z| \leq 1  \right\}.
\end{equation}

The trace map $f$ restricted to $\mathbb{S}$ is a factor of the hyperbolic automorphism of $\mathbb{T}^2 \eqdef \mathbb{R}^2 / \mathbb{Z}^2$ given by
\begin{equation}\label{eq:anosov}
\mathcal{A}( \theta, \varphi ) = (\theta + \varphi, \theta).
\end{equation}
The semi-conjugacy is given by the map
\begin{equation}\label{eq:sconj}
F : ( \theta, \varphi ) \mapsto ( \cos 2\pi(\theta + \varphi), \cos 2\pi \theta, \cos 2\pi \varphi ).
\end{equation}
Let us fix a Markov partition of the map $\mathcal{A}$ (for example, see Figure~\ref{fig:markov}). Pick one element of this Markov partition, denote it by $\mathcal{K}$, and let $\mathcal{K}_0$ denote the projection of $\mathcal{K}$ to $\mathbb{S}$ via the map $F$. The Markov partition for $f: \mathbb{S} \to \mathbb{S}$ can be continued to a Markov partition for the map $f|_{\Omega_V}: \Omega_V \to \Omega_V$. One can check that the elements of this induced Markov partition on $\Omega_V$ are disjoint for every $V > 0$. Let us denote the continuation of $\mathcal{K}_0$ along the parameter $V$ by $\mathcal{K}_V$.

Suppose that $\sigma : \tilde{\mathcal{E}} \to \tilde{\mathcal{E}}$
is the two-sided topological Markov chain conjugate to $f|_{\Omega_V}$ via the conjugacy
$H_{V} : \tilde{\mathcal{E}} \to \Omega_{V}$.
Let $\tilde{\mu}$ be the measure of maximal entropy for
$\sigma : \tilde{\mathcal{E}} \to \tilde{\mathcal{E}}$, and let $\tilde{\mu}_V$ denote the induced measure on $\Omega_V$, that is, $\tilde{\mu}_{V} = H_{V}(\tilde{\mu})$.

Consider the stable manifolds which lie inside $\mathcal{K}_{0}$. These stable manifolds have natural continuations for all $V > 0$ (see \cite[Proposition 4.10]{Yessen2011} for details), and so they form a lamination of two-dimensional smooth (away from $S_0$) manifolds. We denote the resulting lamination of these \emph{center-stable manifolds} by $\widetilde{\Omega}$ (these are the same center-stable manifolds as in Section~\ref{sec:background} and the proof of Theorem~\ref{thm:thm2} above). Let us restrict the measure $\tilde{\mu}_{V}$ to $\mathcal{K}_V$ and normalize it. This naturally induces a measure on $\widetilde{\Omega}$; we abuse the notation and denote this measure by $\widetilde{\mu}$.

\begin{prop}\label{prop:meas-proj}
For every $\lambda \in \mathcal{R}$, the projection of $\widetilde{\mu}$ to $\ell_\lambda$ is proportional to the density of states measure
$dk_\lambda$ under the identification \eqref{eq:line}. In particular, if $\mathcal{K}$ is
the element of the Markov partition in Figure \ref{fig:markov} which contains the interval $[0, 1/2] \times \{0\}$, then the projection
of $\widetilde{\mu}$ to $\ell_\lambda$ corresponds to
the density of states measure under the identification \eqref{eq:line}.
\end{prop}

For a proof of Proposition~\ref{prop:meas-proj}, see Claim~3.16 and the discussion preceding it in \cite{Yessen2011a}, which is an extension to the Jacobi case of the original result given in \cite{Damanik2012} for the Schr\"odinger operators.


Let $\ell_{(\mathfrak{p}, \mathfrak{q})}$ be as in \eqref{eq:line}, let $J$ be a compact interval in $\R$,
and assume that $\mathcal{U}$ is a subset of $\mathcal{R}$, where $\mathcal{R}$ is as in \eqref{eq:params}.
Assume further that $\alpha: J\rightarrow\mathcal{U}$ is analytic, so that
$\set{\ell_{\alpha(t)}}_{t\in J}$ is an analytic family of lines.

Let $\mathcal{E}^{+}$ be the one-sided topological Markov chain which is obtained by deleting the negative side of
$\tilde{\mathcal{E}}$.
Let $\mathcal{E}$ be the subset of $\mathcal{E}^{+}$ which corresponds to $\widetilde{\Omega}$.
Note that $\widetilde{\mu}$ naturally induces a probability measure on $\mathcal{E}$, which we denote by $\mu$.

For any $\omega \in \mathcal{E}$,  we denote the corresponding leaf of the center-stable manifold by $\widetilde{\omega}$.
Let $\pi_{\alpha(t)} : \widetilde{\Omega} \to \ell_{\alpha(t)}$ be the map mapping each center-stable manifold to the point of its intersection with
$\ell_{\alpha(t)}$.
Let us define $\Pi_{\alpha(t)} : \mathcal{E} \to \mathbb{R}$ by
\begin{equation*}
\Pi_{\alpha(t)}(\omega) = \ell_{\alpha(t)}^{-1} \circ \pi_{\alpha(t)}(\widetilde{\omega}).
\end{equation*}
Set $\nu_{\alpha(t)} = \Pi_{\alpha(t)}(\mu)$.
Note that the measure $\nu_{\alpha(t)}$ is a probability measure supported on the spectrum $\Sigma_{\alpha(t)}$.

For any $\omega \in \mathcal{E}$ and $t \in J$, if we have
$\pi_{\alpha(t)}(\widetilde{\omega}) \in S_{V}$ for some $V > 0$,
then we sometimes express this dependency explicitly and write this $V$ as
$V(\omega, t)$ below.



Let us denote by $\Lyap^u(\mu_V)$ the unstable Lyapunov exponent of $f|_{S_V}$ with respect to the measure $\mu_V$. The following statement can be proved by applying the proofs of Propositions~3.8 and 3.9 of \cite{Damanik2013X} without modification.



\begin{prop}\label{key_prop}
Assume that for all $t\in J$ and $\omega \in \mathcal{E}$, $\abs{\frac{d}{dt} \Lyap^u(\mu_{V(\omega, t)})}\geq\delta > 0$. Then for every $\epsilon > 0$, there exist $N^{\ast} \in \mathbb{Z}_{+}$ and a set $\mathcal{E}_{0} \subset \mathcal{E}$
such that $\mu( \mathcal{E}_0 ) > 1 - \frac{\epsilon}{2}$, and such that for $t \in J$,
$\omega \in \mathcal{E}_{0}$, and $N \geq N^{*}$, we have
\begin{equation*}
\lim_{n \to \infty} \frac{1}{n} \log \left\| Df^{n}(\pi_{\alpha(t)}(\widetilde{\omega})) |_{\ell_{\alpha(t)}} \right\|
= \Lyap^{u}(\mu_{V(\omega, t)}),
\end{equation*}
and
\begin{equation*}
\left|  \frac{d}{d t} \left( \frac{1}{N} \log \left\| Df^{N}( \pi_{\alpha(t)}(\widetilde{\omega}) ) |_{\ell_{\alpha(t)}} \right\| \right)  \right|
> \frac{\delta}{4}.
\end{equation*}
\end{prop}

We can now prove
\begin{prop} \label{prop:dgs:acconv}
Let $\eta$ be a compactly supported Borel measure on $\mathbb{R}$ such that $\lhdim(\eta, x) \geq d_\eta$ for $\eta$-a.e.\ $x$. Assume that $d_{\eta} + \dim_{H} \nu_{\alpha(t)} > 1$ for all $t \in J$. Assume also that there exists $\delta > 0$ such that $\abs{\frac{d V(\omega, t) }{dt}} > \delta$ for all $\omega \in \mathcal{E}$ and $t \in J$. Then $\eta \ast \nu_{\alpha(t)}$ is absolutely continuous with respect to Lebesgue measure for a.e. $t \in J$.
\end{prop}

\begin{proof}
The proof is essentially the same as the proof of \cite[Theorem~3.7]{Damanik2013X}. We need to check that the conditions in Proposition~\ref{key_prop} are satisfied. According to the hypothesis of Proposition \ref{prop:dgs:acconv}, for all $\omega\in \mathcal{E}$, $V(\omega, t)$ is strictly monotone with uniformly (in $t$) nonzero derivative. On the other hand, $\mathrm{Lyap}^u(\mu_{V(\omega, t)})$ is analytic in $V$ (this follows from general principles, since $f_V = f|_{S_V}$ is analytic in $V$). It follows that for any given $\omega$, away from a set of nonaccumulating points $t$, the derivative $\frac{d}{dt}\mathrm{\Lyap}^u(\mu_{V(\omega, t)})$ is nonzero. By continuity, using compactness of $\mathcal{E}$, this can be ensured for all $\omega$ by restricting $t$ to a sufficiently small nontegenerate interval.

Let $\mathcal{E}_0 \subset \mathcal{E}$ and $N^{*} \in \mathbb{Z}_+$ be as in Proposition \ref{key_prop}.  Let us take $\omega, \omega' \in \mathcal{E}_0$ in such a way that $p(t) = \pi_{\alpha(t)}(\widetilde{\omega})$ and $q(t) = \pi_{\alpha(t)}(\widetilde{\omega}')$ are sufficiently close. Let $n \in \mathbb{Z}_+$ be such that the distance between $f^{n}(p)$ and $f^{n}(q)$ is of order one. Let us introduce coordinates on each curve $\ell_{\alpha(t)}$, $f(\ell_{\alpha(t)}), \cdots, f^{n}(\ell_{\alpha(t)})$, using the original parametrization and taking $p \in \ell_{\alpha(t)}$, $f(p) \in f(\ell_{\alpha(t)}), \cdots, f^{n}(p) \in f^{n}(\ell_{\alpha(t)})$
to be the origin. We consider the map $f^{-1} : f^{i}(\ell_{\alpha(t)}) \to f^{i-1}(\ell_{\alpha(t)})$ in these coordinates, and write the resulting map by $k^{(i)}_{\alpha(t)} : \mathbb{R} \to \mathbb{R}$. Let $l^{(i)} = \frac{\partial k^{(i)}_{\alpha(t)}}{\partial x} (0)$.

Then, by the Lemma 3.11 of \cite{Damanik2013X} and by the Proposition \ref{key_prop}, we have
\begin{equation*}
\frac{d}{dt} \dist(p(t), q(t)) \leq \left( \prod_{s=1}^{n} l^{(s)} \right)(C - \delta'n)
\end{equation*}
for some $C, \delta' > 0$. The rest of the argument is a verbatim repetition of the proof of \cite[Theorem~3.7]{Damanik2013X} with our Proposition~\ref{prop:key_prop} in place of \cite[Proposition~2.3]{Damanik2013X} and Proposition~\ref{key_prop} in place of \cite[Propositions~3.8 and 3.9]{Damanik2013X}.

\end{proof}

We now prove the following measure-theoretic analog of Lemma \ref{lem:app-4}.
\begin{lem}\label{lem:Jake}
For every $\Delta, \delta \in (0, 1)$ there exists a nonempty open set $\mathcal{U} \subset \mathcal{R}$ such that we have the following for every $\lambda \in \mathcal{U}$.

\begin{enumerate}\itemsep0.5em

\item[\textup{(1)}] There exists $E_0\in\Sigma_\lambda$ such that $\lhdim(dk_\lambda, E) > \Delta$ for all $E \in \Sigma_\lambda$ in a sufficiently small neighborhood of $E_0$  {\rm(}where $\lhdim(dk_\lambda, E_0)$ is as defined in \eqref{dimension} with $dk_\lambda$ and $E_0$ in place of $\eta$ and $x${\rm)}.

\item[\textup{(2)}] At one of the extrema of $\Sigma_\lambda$, $E_\mathrm{end}$, we have $\lhdim(\Sigma_\lambda, E_\mathrm{end}) < \delta$.
\end{enumerate}
\end{lem}
\begin{proof}

First, the existence of $\lhdim(dk_\lambda, E)$ for $dk_\lambda$-a.e. $E\in \Sigma_\lambda$, that is, the existence of the limit in \eqref{dimension}, is proved in \cite[Proposition~3.15]{Yessen2011a}. We can now apply the proof of \cite[Theorem~2.6]{Yessen2011a} without modification to obtain the following (in \cite{Yessen2011a} the proof is given for the entire set $\Sigma_\lambda$ with $\lambda$ close to $(1, 0)$; however, the same proof gives the following local result).
\begin{claim}\label{c:Jake-2}
For all $\Delta \in (0,1)$ there exists $\epsilon > 0$ such that for all $\lambda\in \mathcal{R}$, if $E_0\in\Sigma_\lambda$ is such that $0 < I(\ell_\lambda(E_0)) = V_\lambda(E_0) < \epsilon$ {\rm(}with $V_\lambda$ as in \eqref{eq:v}{\rm)}, then $\lhdim(dk_\lambda, E) > \Delta$ for $dk_\lambda$-a.e. $E$ in a sufficiently small neighborhood of $E_0$ {\rm(}any open neighborhood of $E_0$ is necessarily of nonzero $dk_\lambda$ measure{\rm)}.
\end{claim}

At the same time, just as in the proof of Lemma~\ref{lem:app-4}, we know that there exists $C > 0$ such that if $\lambda \in \mathcal{R}$ such that for one of the extrema of $\Sigma_\lambda$, $E_\mathrm{end}$, $V_\lambda(E_\mathrm{end})> C$, then $\lhdim(\Sigma_\lambda, E_\mathrm{end}) < \delta$.

Now, let $\epsilon > 0$ as in Claim~\ref{c:Jake-2} and $C > 0$ as in the preceding paragraph. Just as in the proof of Lemma \ref{lem:app-4}, let $\mathcal{U}\subset\mathcal{R}$ be an open set such that for all $\lambda\in\mathcal{U}$, there exists $E_0\in\Sigma_\lambda$ and one of the extrema $E_\mathrm{end}$ of $\Sigma_\lambda$ satisfying $V_\lambda(E_0)\in(0, \epsilon)$ and $V_\lambda(E_\mathrm{end}) > C$.
\end{proof}

\begin{rem} Just as in Lemma \ref{lem:app-4}, $\mathcal{U}$ can be adjusted so that $E_0$ can be taken to be the other extremum of the spectrum (see Remark~\ref{rem:extremum}).
\end{rem}

In what follows, we fix a nonempty open set $\mathcal{U}\subset \mathcal{R}$ as in Lemma~\ref{lem:Jake} with $\Delta > \frac{1}{2}$, and so that $(1, 0)\notin \mathcal{U}$. Pick and fix $\lambda \in \mathcal{U}$. We see that for every $\tilde{\lambda} \in \mathcal{U}$, the singular continuous component of $dk_{\lambda} * dk_{\tilde{\lambda}}$ is nonzero. Indeed, since by Lemma~\ref{lem:Jake} (see also Lemma~\ref{lem:app-4}), $\Sigma_{(\lambda, \tilde{\lambda})}^2$ contains a Cantor set $\mathcal{C}$ of zero Lebesgue measure which is the sum of Cantor subsets of $\Sigma_{\lambda}$ and $\Sigma_{\tilde{\lambda}}$ of nonzero $dk_{\lambda}$ and $dk_{\tilde{\lambda}}$ measure, respectively, we have $(dk_{\lambda} * dk_{\tilde{\lambda}})(\mathcal{C}) > 0$.

Now fix any $\tilde{\lambda} \in \mathcal{U}$. Without loss of generality, let us assume that $\mathcal{U}$ is an open ball in $\mathcal{R}$. Let $\alpha: [0, 1]\rightarrow \mathcal{R}$ be an analytic curve with the following properties.
\it
\begin{enumerate}\itemsep0.5em

\item[\textup{(1)}] $\alpha(0) = (1, 0)$ and for all $t\in (0, 1]$, $\alpha(t)\neq (1, 0)$.

\item[\textup{(2)}] $\alpha(1) = \lambda$.

\item[\textup{(3)}] For some $t_0 \in (0, 1)$, $\alpha(t_0) = \tilde{\lambda}$.

\item[\textup{(4)}] The arc $\alpha([t_0, 1])$ is contained entirely in $\mathcal{U}$.
\end{enumerate}
\rm

Now take $E \in \Sigma_{\alpha(0)}$ and its continuation $E(t)$ along $t \in (0, 1]$. Appealing again to the analyticity of the center-stable manifolds (proof of Theorem~\ref{thm:transversality}) and to transversality (Theorem~\ref{thm:transversality}), we have that $E(t)$ is analytic on $(0, 1]$; therefore, $V_{\alpha(t)}(E(t))$ is also analytic. This implies

\begin{lem}\label{lem:v-non-constant}
Let $E(t)$ be as above, and assume in addition that for all $t$, $E(t)$ is not an extremum of $\Sigma_{\alpha(t)}$. Then $V_{\alpha(t)}'(E(t)) = \frac{dV_{\alpha(t)}(E(t))}{dt}$ admits at most finitely many zeros in the interval $[t_0, 1]$.
\end{lem}

\begin{proof}
Assume for the sake of contradiction that $V'$ admits infinitely many zeros in $[t_0, 1]$. By analyticity, this means that $V_{\alpha(t)}(E(t))$ is constant on $[t_0, 1]$. Using analyticity again, it then must be constant on $(0, 1]$. Since $E(t)$ is not an extremum of $\Sigma_{\alpha(t)}$, we have $V_{\alpha(t)}(E(t)) > 0$ by monotonicity of $V_{\alpha(t)}$ in $E$, which we know from \eqref{eq:v}. Consequently, $\abs{E(t)}\rightarrow \infty$ as $t\rightarrow 0$, since $\ell_{\alpha(t)}\rightarrow \ell_{\alpha(0)}$ in the $C^1$ topology as $t\rightarrow 0$, and $\ell_{\alpha(0)}$ lies in $S_0$. On the other hand, by compactness of $[0, 1]$ and continuity (in operator norm topology) of the map $t \mapsto H_{\alpha(t)}$, $\Sigma_{\alpha(t)}$ is uniformly bounded for $t\in [0, 1]$.
\end{proof}

\begin{proof}[Proof of Theorem~\ref{thm:thm1}]
Now with $E_0 \in \Sigma_{\lambda}$ as in Lemma~\ref{lem:Jake}, take $E\in \Sigma_{\lambda}$ in a neighborhood of $E_0$ so that $E$ is not an extremum of $\Sigma_{\lambda}$ and (1) of Lemma~\ref{lem:Jake} is satisfied. By Lemma~\ref{lem:v-non-constant}, there exist finitely many points $\set{p_1, \dots, p_n}\subset [t_0, 1]$ such that $[t_0, 1]\setminus\set{p_1, \dots, p_n}$ can be written as a union of compact nondegenerate intervals $\set{J_n}_{n\in\N}$ such that there exists a sequence $\set{\delta_n > 0}_{n\in\N}$ such that for all $t\in J_n$, $V'_{\alpha(t)}(E(t)) > \delta_n$. By continuity it follows that for all $n$ there exist $E_n^+, E_n^- \in \Sigma_{\lambda}$ with $E \in (E_n^+, E_n^-)$, such that for all $t\in J_n$ and $\tilde{E}\in \Sigma_{\lambda}\cap [E_n^-, E_n^+]$, $\tilde{E}(t) \in [E_n^-(t), E_n^+(t)]$ and $V'_{\alpha(t)}(\tilde{E}(t)) > \delta_n$ (here $\tilde{E}(t)$ and $E_n^\pm(t)$ are defined similarly to $E(t)$ as the continuations of $\tilde{E}$ and $E_n^\pm$, respectively).

Now let us replace $\mathcal{E}$ in Proposition~\ref{key_prop} with $\mathcal{E}_n \subset \mathcal{E}$ such that $\Pi_{\alpha(1)}(\mathcal{E}_n)\subset \Sigma_{\lambda}\cap [E_n^-, E_n^+]$ so that $\Pi_{\alpha(t)}(\mathcal{E}_n)\subset \Sigma_{\alpha(t)}\cap [E_n^-(t), E_n^+(t)]$ and $dk_{\alpha(t)}(\Pi_{\alpha(t)}(\mathcal{E}_n)) > 0$, $J$ with $J_n$, $\delta$ with $\delta_n$, $\eta$ with $dk_{\lambda}$, and $\nu_{\alpha(t)}$ with $dk_{\alpha(t)}$ restricted to $[E_n^+(t), E_n^-(t)]$ and normalized. We can do this by taking a sufficiently fine refinement of the original Markov partition so that $\nu_{\alpha(t)}$ is a normalized restriction of $dk_{\alpha(t)}$ to the intersection of $\Sigma_{\alpha(t)}$ with a compact interval of nonzero $dk_{\alpha(t)}$ measure. Then from Proposition \ref{prop:dgs:acconv} we have $dk_{\lambda} * dk_{\alpha(t)} \ll \mathcal{L}^1$ for Lebesgue almost every $t \in J_n$. Since this holds for every analytic curve $\alpha$ which satisfies items (1)--(4), Theorem~\ref{thm:thm1} follows.
\end{proof}

\begin{rem}\label{rem:thm2-extension}
It is now easy to see, combining Lemmas \ref{lem:Jake} and \ref{lem:app-4} that there exists a nonempty open set $\mathcal{U}\subset \mathcal{R}$ that satisfies the conclusions of both Theorem \ref{thm:thm2} and Theorem \ref{thm:thm1}.
\end{rem}

\subsection{Proof of Theorem \ref{thm:thm3}}\label{sec:proof-thm3}

Damanik, Fillman, and Gorodetski have computed the curve of initial conditions for this version of the continuum Fibonacci Hamiltonian in \cite{DFG2014}.

\begin{align}\label{eq:curve-continuum}
\begin{split}
x(E)
& =
\cos\sqrt E, \hspace{5mm}
y(E)
=
\cos\sqrt{E-\lambda}, \\
z(E)
& =
\cos\sqrt E \cos \sqrt{E-\lambda}
-\frac{1}{2} \left( \sqrt{\frac{E}{E-\lambda}} + \sqrt{\frac{E-\lambda}{E}} \right)
\sin\sqrt E \sin \sqrt{E-\lambda}.
\end{split}
\end{align}
In particular, $\ell_\lambda(E) = (x(E), y(E), z(E))$ is an analytic curve in $\R^3$ and $\C^3$ both, in $E$ and in $\lambda$, for each $\lambda > 0$. Using the expressions for $x$, $y$, and $z$, we can compute the Fricke-Vogt invariant as a funciton of $E$. We get
\begin{equation}\label{eq:cont:invariant}
I(E,\lambda)
\eqdef
I(x(E), y(E), z(E))
=
\frac{\lambda^2}{4E(E-\lambda)} \sin^2 \sqrt E \sin^2 \sqrt{E-\lambda}.
\end{equation}
Notice that the expressions in \eqref{eq:curve-continuum} and \eqref{eq:cont:invariant} are analytic in $E$ and in $\lambda$, with removable singularities at $E = 0$ and $E = \lambda$. Note also that for all $\lambda$, $I(E, \lambda) \to 0$ as $E \to \infty$, so the analytic curve in \eqref{eq:curve-continuum} approaches the Cayley cubic in the high-energy regime. As before, the proof of Theorem~\ref{thm:thm3} relies on two pieces:
\begin{enumerate}\itemsep0.5em
\item Near the bottom, the spectrum is thin in the sense of Hausdorff dimension for large enough coupling.
\item For any fixed coupling, the spectrum has large local thickness at sufficiently high energies.
\end{enumerate}

We can use \cite{DFG2014} to control the Hausdorff dimension of the spectrum near the bottom. First, we need some control on the ground state energy.

\begin{lem} \label{lem:ground}
Denote $E_0(\lambda) = \inf(\Sigma_\lambda)$. There exists a constant $C \in (0, 3)$ such that $0 \le E_0(\lambda) \le C$ for all $\lambda \ge 0$.
\end{lem}

\begin{proof}
Since $V_\lambda(x) \ge 0$ for all $x \in \R$ and all $\lambda \ge 0$, the inequality $E_0(\lambda) \ge 0$ follows immediately. Notice that there is an interval $I \subset \R$ of length two such that $V_\lambda$ vanishes on $I$ for all $\lambda \ge 0$ (obviously, there are infinitely many such intervals). Let $\varphi$ be a smooth function with compact support contained in $I$ such that $\| \varphi \|_2 = 1$. Then
$$
\langle \varphi, H_\lambda \varphi \rangle
=
\int \! \overline{\varphi} (-\varphi'' + V_\lambda\varphi)
=
\int \! |\varphi'|^2
<
\infty.
$$
The second equality follows from $V|_{\mathrm{supp}(\varphi)} \equiv 0$ and integration by parts. Since $H_\lambda$ is self-adjoint, we have
$$
\inf \Sigma_\lambda
=
\inf_{\|\psi\| = 1} \langle \psi, H\psi \rangle
\leq
\langle \varphi, H_\lambda \varphi \rangle,
$$
so we may take $C = \| \varphi'\|^2_2$. By choosing $\varphi$ to be a smooth function which is suitably close to a (the square root of) a tent function on an interval of length two, we see that we can make $\| \varphi'\|_2^2 < 3$.
\end{proof}

\begin{lem} \label{l:bigl:thin}
In the large coupling regime, we have
$$
\lim_{\lambda \to \infty}
\dim_{\Hd}^{\loc} (\Sigma_\lambda, E_0(\lambda))
=
0.
$$
\end{lem}

\begin{proof}
This follows from \cite[Theorem~6.5]{DFG2014}, \eqref{eq:cont:invariant}, and Lemma~\ref{lem:ground}. Notice that we need the $C$ from Lemma~\ref{lem:ground} to be bounded away from $\pi^2$ to effectively bound $I(E_0(\lambda),\lambda)$ from below as $\lambda \to \infty$.
\end{proof}


\begin{rem}\label{rem:accum}
As mentioned in the introduction, it follows directly from \cite{DFG2014} that for all $\lambda \geq 0$, $\hdim(\Sigma_\lambda) = 1$. It also follows from \cite{DFG2014} that $\hdim$ accumulates at infinity (that is, for all $E > 0$, $\hdim([E, \infty)\cap \Sigma_\lambda)=1$). On the other hand, there may exist other points where the dimension accumulates. Indeed, $E\in\Sigma_\lambda$ is such a point if and only if $\ell_\lambda(E)\in S_0$ (see \cite[Section~2]{Damanik2013a}). Now, if $\lambda$ is of the form $4\pi^2(a^2 - b^2)$ for $a, b \in \N$ with $a > b$, then with $E = 4a^2\pi^2$, it is evident from \eqref{eq:curve-continuum} that $\ell_\lambda(E) = (1, 1, 1)$, which is a point on $S_0$ that is fixed under the action by $f$. Thus we have $\lhdim(\Sigma_\lambda, E) = 1$.
\end{rem}

\begin{lem} \label{lem:biglam:thin}
For all $\lambda_1, \lambda_2 > 0$ sufficiently large, there exists an interval $J$ such that $J \cap \Sigma_{(\lambda_1,\lambda_2)}^2$ is a {\rm(}nonempty{\rm)} Cantor set.
\end{lem}

\begin{proof}

Take $\lambda_1, \lambda_2$ so that $\lhdim(\Sigma_{\lambda_i}, E_0(\lambda_i)) < \frac{1}{4}$. Observe that $E_0 \eqdef E_0(\lambda_1)+E_0(\lambda_2)$ marks the bottom of $\Sigma_{(\lambda_1, \lambda_2)}^2$. It follows that for all $\epsilon > 0$ sufficiently small, there exist $\delta_i > 0$ and sets
\begin{align*}
J_i^{(k)}(\delta_i)
\subset
\Sigma_{\lambda_i}\cap (E_0(\lambda_i), E_0(\lambda_i)+\delta_i),
\quad
i = 1, 2, \, k \in \N,
\end{align*}
such that
\begin{align}\label{eq:union}
\Sigma_{(\lambda_1, \lambda_2)}^2\cap(E_0, E_0+\epsilon) \subseteq \bigcup_{k\in\N} \left[J_1^{(k)}(\delta_1)+J_2^{(k)}(\delta_2)\right].
\end{align}

It may happen that $\ell_{\lambda_i}$ intersects a center-stable manifold tangentially at the point $\ell_{\lambda_i}(E_0(\lambda_i))$; however, since tangencies cannot accumulate (see \cite[Section 3.2.1]{Yessen2011a}), there exist $\epsilon_1, \epsilon_2 > 0$ such that for all $E_i \in \Sigma_{\lambda_i}\cap (E_0(\lambda_i), E_0(\lambda_i) + \epsilon_i)$, the intersection of $\ell_{\lambda_i}$ with a center-stable manifold at the point $\ell_{\lambda_i}(E_i)$ is transversal. It follows from \cite[Proposition 3.2]{Yessen2011a} that for all $i$ and $k$ such that $J_i^{(k)} \subset \Sigma_{\lambda_i}\cap (E_0(\lambda_i), E_0(\lambda_i) + \epsilon_i)$, $\hdim(J_i^{(k)}) = \bdim(J_i^{(k)})$.

Now choose $\epsilon_i > 0$ in the previous paragraph so small, that in addition to the transversality condition being satisfied, we have
\begin{align*}
\hdim(\Sigma_{\lambda_i}\cap (E_0(\lambda_i), E_0(\lambda_i) + \epsilon_i)) < \frac{1}{4}.
\end{align*}
Take $\epsilon > 0$ above sufficiently small, ensuring that we can take $\delta_i \in (0, \epsilon_i)$. Then just as in \eqref{eq:haus-box} we have
\begin{align*}
\hdim(J_1^{(k)}(\delta_1) + J_2^{(k)}(\delta_2)) \leq \hdim(J_1^{(k)}(\delta_1)) + \hdim(J_1^{(k)}(\delta_2)) < \frac{1}{2}\text{ for all } k.
\end{align*}
But then from \eqref{eq:union} we have $\hdim(\Sigma_{(\lambda_1, \lambda_2)}^2\cap (E_0, E_0 + \epsilon)) \leq \frac{1}{2} < 1$, ensuring that $\Sigma_{(\lambda_1, \lambda_2)}^2$ is a Cantor set near $E_0$.
\end{proof}

\begin{rem}
We do not suspect that tangencies between $\ell_\lambda$ and the center-stable manifolds occur; however, a proof of absence of tangencies would introduce unnecessary technical difficulties. For this reason, we have decided to omit it.
\end{rem}

To prove that the interior of $\Sigma_\lambda^2$ is nonempty for all $\lambda > 0$, we first prove the following lemma (in what follows, $\mathbb{S}$, $\mathcal{A}$, and $F$ are as in \eqref{eq:cc-center}, \eqref{eq:anosov}, and \eqref{eq:sconj}).

\begin{figure}[t!]
\subfigure[$\lambda = 50$, $E\in(51, 6\cdot 10^3)$.]{
\includegraphics[scale=0.45]{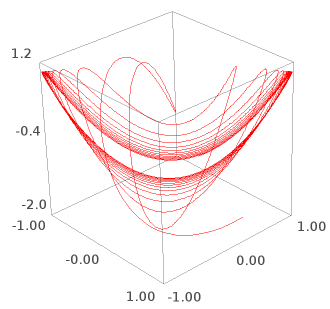}
}
\quad
\subfigure[$\lambda = 50$, $E\in(6\cdot 10^4, 8\cdot 10^4)$.]{
\includegraphics[scale=0.45]{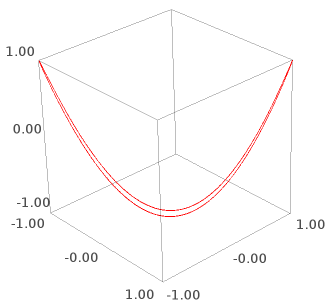}
}
\caption{}\label{fig:curve}
\end{figure}

\begin{lem} \label{l:bigE:thick}
For all $\lambda > 0$ and $\tau_0 > 0$ there exists $M > 0$ and infinitely many $E \in (M, \infty)\cap \Sigma_\lambda$, such that $\tau^{\loc}(\Sigma_\lambda, E) > \tau_0.$
\end{lem}

\begin{proof}
Observe that $\ell_0$ passes through the point $(0, 0, -1)$, which is periodic. Its continuation to the surfaces $S_V$, $V\neq 0$, is of the form $(0, 0, -\sqrt{V+1})$. The center-stable manifold containing the point $(0, 0, -1)$ is transversal to $\mathbb{S}$; denote this manifold by $W$.

It is easy to see that $\ell_0$ lies in $\mathbb{S}$, and that $F^{-1}(\ell_0)$ is horizontal in $\T^2$. On the other hand, since the eigendirections of $\mathcal{A}$ do not align with $(1, 0)$, $F^{-1}(\ell_0)$ is transversal to the stable foliation of $\mathcal{A}$ on $\T^2$. Now by \cite[Lemma 3.1]{Damanik2009} we conclude that $\ell_0$ is transversal to the stable foliation of $f$ on $\mathbb{S}$. It follows that $\ell_0$ is transversal to $W$ at the point $(0, 0, -1)$.

Let $\dist(A, B)$ denote the distance between the sets $A$ and $B$. It is easy to see from the expression for $\ell_\lambda$ that away from a neighborhood of the points $(1, 1, 1)$ and $(-1, -1, 1)$, $\ell_\lambda((E_\mathrm{end},\infty))$ approaches $\ell_0(\R)$ in the $C^1$ topology as $E_\mathrm{end}\rightarrow \infty$ (see Figure \ref{fig:curve}). It follows that there exists $\delta \in(0, \frac{\pi}{2})$ such that for all $\lambda > 0$ there exists $\epsilon_0 > 0$, such that for every $\epsilon \in (0, \epsilon_0)$, every segment $L$ of $\ell_\lambda(\R)$ that satisfies $\dist(L, \set{(0, 0, -1)}) < \epsilon$ intersects the center-stable manifolds uniformly transversally at an angle of size at least $\delta$. Now the method of (the proof of) Theorem \ref{thm:thm2} applies and yields $\tau^\loc(\Sigma_\lambda, E) > \tau_0$ for infinitely many sufficiently large $E$.
\end{proof}

Now with $\tau_0$ in Lemma \ref{l:bigE:thick} chosen in $(1, \infty)$, it follows, just as in the proof of Theorem \ref{thm:thm2}, that for all $\lambda_i > 0$, $i=1,2$, $\Sigma_{(\lambda_1, \lambda_2)}^2$ contains an interval. Combining this with Lemma~\ref{lem:biglam:thin}, we obtain Theorem \ref{thm:thm3}.

\begin{rem}\label{rem:thm3-details}
As mentioned in Remark \ref{rem:thm3} (2), we suspect that for all $\lambda_i > 0$, $i=1,2$, there exists $E_i\in \Sigma_{\lambda_i}$ such that $[E_1 + E_2, \infty)\subset\Sigma_{(\lambda_1, \lambda_2)}^2$, which is obviously stronger than $\Sigma_{(\lambda_1, \lambda_2)}^2$ just having nonempty interior. A proof of this would involve control of the global thicknes of $\Sigma_\lambda\cap [E, \infty)$ as $E \rightarrow \infty$. The global thickness could be controlled via finer estimates than those used for the control of the local thickness in the proofs above, but the arguments would be far more technical. For this reason we have decided to delegate this task to later investigations.
\end{rem}

\section*{Acknowledgements}

The authors thank David Damanik and Anton Gorodetski for helpful discussions. J.\ F.\ thanks the Isaac Newton Institute for Mathematical Sciences for their hospitality during the Programme on Periodic and Ergodic spectral problems, during which portions of this work were undertaken.

\bibliographystyle{amsplain}
\bibliography{bibliography}

\providecommand{\bysame}{\leavevmode\hbox to3em{\hrulefill}\thinspace}
\providecommand{\MR}{\relax\ifhmode\unskip\space\fi MR }
\providecommand{\MRhref}[2]{%
  \href{http://www.ams.org/mathscinet-getitem?mr=#1}{#2}
}
\providecommand{\href}[2]{#2}
\begin{thebibliography}{10}

\bibitem{Astels2000}
S.~Astels, \emph{{Cantor sets and numbers with restricted partial quotients}},
  Trans. Amer. Math. Soc. \textbf{352} (2000), 133--170.

\bibitem{Buzzard2001}
G.~T. Buzzard and K.~Verma, \emph{Hyperbolic automorphisms and holomorphic
  motions in $\mathbb{C}^2$}, Michigan Math. J. \textbf{49} (2001), 541--565.

\bibitem{Cantat2009}
S.~Cantat, \emph{{Bers and H{\'e}non, Painlev{\'e} and Schr{\"o}dinger}}, Duke
  Math. J. \textbf{149} (2009), 411--460.

\bibitem{Casdagli1986}
M.~Casdagli, \emph{{Symbolic dynamics for the renormalization map of a
  quasiperiodic Schr{\"o}dinger equation}}, Commun. Math. Phys. \textbf{107}
  (1986), 295--318.

\bibitem{Damanik2013}
D.~Damanik, M.~Embree, and A.~Gorodetski, \emph{{Spectral properties of
  Schr{\"o}dinger operators arising in the study of quasicrystals}}, preprint
  (arXiv:1210.5753).

\bibitem{Damanik2008}
D.~Damanik, M.~Embree, A.~Gorodetski, and S.~Tcheremchantsev, \emph{{The
  fractal dimension of the spectrum of the Fibonacci Hamiltonian}}, Commun.
  Math. Phys. \textbf{280} (2008), 499--516.

\bibitem{DFG2014}
D.\ Damanik, J.\ Fillman, and A.\ Gorodetski, Ann.\ Henri Poincar\'e
  \textbf{15} (2014), 1123--1144.

\bibitem{Damanik2009}
D.~Damanik and A.~Gorodetski, \emph{{Hyperbolicity of the trace map for the
  weakly coupled Fibonacci Hamiltonian}}, Nonlinearity \textbf{22} (2009),
  123--143.

\bibitem{Damanik2010}
\bysame, \emph{{Spectral and Quantum Dynamical Properties of the Weakly Coupled
  Fibonacci Hamiltonian}}, Commun. Math. Phys. \textbf{305} (2011), 221--277.

\bibitem{Damanik2012}
\bysame, \emph{{The density of states measure of the weakly coupled Fibonacci
  Hamiltonian}}, Geom. Funct. Anal. \textbf{22} (2012), 976--989.

\bibitem{Damanik2013X}
D.~Damanik, A.~Gorodetski, and B.~Solomyak, \emph{{Absolutely continuous
  convolutions of singular measures and an application to the square Fibonacci
  Hamiltonian}}, preprint (arXiv:1306.4284).

\bibitem{Damanik2014e}
D.~Damanik, A.~Gorodetski, and W.~Yessen, \emph{{The Fibonacci Hamiltonian}},
  preprint (arXiv:1403.7823).

\bibitem{Damanik2013a}
D.~Damanik, P.~Munger, and W.~N. Yessen, \emph{{Orthogonal polynomials on the
  unit circle with Fibonacci Verblunsky coefficients, I. The essential support
  of the measure}}, J. Approx. Theory \textbf{173} (2013), 56--88.

\bibitem{Damanik2013b}
\bysame, \emph{{Orthogonal polynomials on the unit circle with Fibonacci
  Verblunsky coefficients, II. Applications.}}, J. Stat. Phys. \textbf{153}
  (2013), 339--362.

\bibitem{Hasselblatt2002b}
B.~Hasselblatt, \emph{{Handbook of Dynamical Systems: Hyperbolic Dynamical
  Systems}}, vol.~1A, Elsevier B. V., Amsterdam, The Netherlands, 2002.

\bibitem{Hasselblatt2006}
B.~Hasselblatt and Ya. Pesin, \emph{{Partially hyperbolic dynamical systems}},
  Handbook of dynamical systems \textbf{1B} (2006), 1--55, Elsevier B. V.,
  Amsterdam (Reviewer: C. A. Morales).

\bibitem{Ilan2004}
R.~Ilan, E.~Liberty, S.~Even-Dar~Mandel, and R.~Lifshitz, \emph{Electrons and
  phonons on the square fibonacci tilings}, Ferroelectrics \textbf{305} (2004),
  15--19.

\bibitem{LenzSeifStoll2014}
D.~Lenz, C.~Seifert, and P.~Stollman, \emph{{Zero measure Cantor spectra for
  continuum one-dimensional quasicrystals}}, J. Differential Equations
  \textbf{256} (2014), 1905--1926.

\bibitem{Mandel2008b}
S.~Even-Dar Mandel and R.~Lifshitz, \emph{Bloch-like electronic wave functions
  in two-dimensional quasicrystals}, preprint (arXiv:0808.3659).

\bibitem{Mandel2006}
\bysame, \emph{{Electronic energy spectra and wave functions on the square
  Fibonacci tiling}}, Philosophical Magazine \textbf{86} (2006), 759--764.

\bibitem{Mandel2008}
\bysame, \emph{{Electronic energy spectra of square and cubic Fibonacci
  quasicrystals}}, Philosophical Magazine \textbf{88} (2008), 2261--2273.

\bibitem{Mattila1995}
P.~Mattila, \emph{{Geometry of sets and measures in Euclidean spaces}},
  Cambridge studies in mathematics, Cambridge University Press \textbf{44}
  (1995).

\bibitem{Mei201x}
M.~Mei and W.~Yessen, \emph{{Tridiagonal substitution Hamiltonians}}, Math.
  Model. Nat. Phenom. \textbf{5} (2014), 204--238.

\bibitem{Palis1993}
J.~Palis and F.~Takens, \emph{{Hyperbolicity and sensetive chaotic dynamics at
  homoclinic bifurcations}}, Cambridge University Press, Cambridge, 1993.

\bibitem{Roberts1996}
J.~A.~G. Roberts, \emph{{Escaping orbits in trace maps}}, Physica A: Stat.
  Mech. App. \textbf{228} (1996), 295--325.

\bibitem{Roberts1994b}
J.~A.~G. Roberts and M.~Baake, \emph{{The Dynamics of Trace Maps}}, Hamiltonian
  Mechanics: Integrability and Chaotic Behavior, ed. J. Seimenis, NATO ASI
  Series B: Physics (Plenum Press, New York) (1994), 275--285.

\bibitem{Roberts1994}
\bysame, \emph{{Trace maps as 3D reversible dynamical systems with an
  invariant}}, J. Stat. Phys. \textbf{74} (1994), 829--888.

\bibitem{Suto1987}
A.~S{\"u}t\H{o}, \emph{{The spectrum of a quasiperiodic Schr{\"o}dinger
  operator}}, Commun. Math. Phys. \textbf{111} (1987), 409--415.

\bibitem{Teschl1999}
G.~Teschl, \emph{{Jacobi operators and completely integrable nonlinear
  lattices}}, AMS mathematical surveys and monographs \textbf{vol. 72},
  American Mathematical Society, Providence, RI.

\bibitem{Yessen2011a}
W.~N. Yessen, \emph{{Spectral analysis of tridiagonal Fibonacci Hamiltonians}},
  J. Spectr. Theory \textbf{3} (2013), 101--128.

\bibitem{Yessen2011}
\bysame, \emph{{On the spectrum of 1D quantum Ising quasicrystal}}, Annal. H.
  Poincar{\'e} \textbf{15} (2014), 419--467.

\end{thebibliography}

\end{document}